\documentclass[11pt]{article}

\usepackage[margin=1.3in]{geometry}
\usepackage{amssymb,amsmath}
\usepackage{amsfonts}
\usepackage{amsthm}
\usepackage{hyperref, xcolor}
\usepackage{mathrsfs}
\usepackage{textcomp}

\usepackage{verbatim}

\usepackage{sectsty}

\hypersetup{
    colorlinks,%
    citecolor=black,%
    filecolor=black,%
    linkcolor=black,%
    urlcolor=black
}

\makeatletter

\newdimen\bibspace
\renewenvironment{thebibliography}[1]{%
 \section*{\refname 
       \@mkboth{\MakeUppercase\refname}{\MakeUppercase\refname}}%
     \list{\@biblabel{\@arabic\c@enumiv}}%
          {\settowidth\labelwidth{\@biblabel{#1}}%
           \leftmargin\labelwidth
           \advance\leftmargin\labelsep
           \itemsep\bibspace
           \parsep\z@skip     %
           \@openbib@code
           \usecounter{enumiv}%
           \let\p@enumiv\@empty
           \renewcommand\theenumiv{\@arabic\c@enumiv}}%
     \sloppy\clubpenalty4000\widowpenalty4000%
     \sfcode`\.\@m}
    {\def\@noitemerr
      {\@latex@warning{Empty `thebibliography' environment}}%
     \endlist}

\makeatother

\makeatletter

\newtheorem{thm}{Theorem}[section]
\newtheorem{lem}[thm]{Lemma}
\newtheorem{prop}[thm]{Proposition}
\newtheorem{defn}[thm]{Definition}

\newtheorem{cor}[thm]{Corollary}
\newtheorem{rem}[thm]{Remark}

\def\XXint#1#2#3{{\setbox0=\hbox{$#1{#2#3}{\int}$}
  \vcenter{\hbox{$#2#3$}}\kern-.5\wd0}}

\newcommand{\pa}{\partial}              \newcommand{\lda}{\lambda}
\newcommand{\va}{\varepsilon}           \newcommand{\ud}{\mathrm{d}}
\newcommand{\be}{\begin{equation}}      \newcommand{\ee}{\end{equation}}
                 
\newcommand{\Lda}{\Lambda}              \newcommand{\B}{\mathcal{B}}
\newcommand{\R}{\mathbb{R}}              
\newcommand{\al}{\alpha}                 
\newcommand{\om}{\Omega}

\newcommand{\dsup}{\displaystyle\sup}

\begin{document}

\title{Compactness of solutions to nonlocal elliptic equations
 \bigskip}

\author{\medskip   Miaomiao Niu, \  \  Zhipeng Peng, \ \ Jingang Xiong\footnote{Supported in part by NSFC 11501034, a key project of NSFC 11631002, NSFC 11571019.}}

\date{\today}

\maketitle

\begin{abstract}

We show  that all nonnegative solutions of the critical semilinear elliptic equation involving the regional fractional Laplacian  are locally universally bounded. This strongly contrasts with the standard fractional Laplacian case. Second, we consider the fractional critical elliptic equations with nonnegative potentials. We prove compactness of solutions provided the potentials only have non-degenerate zeros. Corresponding to Schoen's Weyl tensor vanishing conjecture for the Yamabe equation on manifolds, we establish a Laplacian vanishing rate of the potentials at blow-up points of solutions.

\end{abstract}

\section{Introduction}

Let $\om$ be an open subset of $\R^n$, $n\ge 2$.  The regional fractional Laplace operator is defined as
\[
(-\Delta_\om)^\sigma u(x):=\mbox{P.V.} c_{n,\sigma}\int_{\om}\frac{u(x)-u(y)}{|x-y|^{n+2\sigma}}\,\ud y \quad \mbox{for }u\in C^2(\om),
\]
where $0<\sigma<1$ is a parameter, $c_{n,\sigma}=\frac{2^{2\sigma}\sigma\Gamma(\frac{n+2\sigma}{2})}{\pi^{\frac{n}{2}} \Gamma(1-\sigma)}$. The regional fractional Laplacian arises, for instance, from  the Feller generator of the reflected symmetric
stable process, see Bogdan-Burdzy-Chen \cite{BBC}, Chen-Kumagai  \cite{CK}, Guan-Ma \cite{GuM}, Guan \cite{Gu}, Mou-Yi \cite{MY} and many others. Here we are interested in universal boundness of positive solutions to nonlinear Poisson equation involving the regional fractional Laplacian. Making use of the standard blow-up argument of Gidas-Spruck \cite{GS} and the Liouville theorem, one can show that any nonnegative solutions of the equation $(-\Delta_\om)^\sigma u(x)=u^{p}$ with $1<p<\frac{n+2\sigma}{n-2\sigma}$ are locally universally bounded. In view of the fractional Sobolev inequality, for $p$ in that range we say the equation is subcritical. In contrast, the critical equation $p=\frac{n+2\sigma}{n-2\sigma}$ has blow-up solutions when $\om=\R^n$. See Jin-Li-Xiong \cite{JLX,JLX1} and references therein for more discussions.

However, if $\om$ has nontrivial complement, we have

\begin{thm}\label{thm:main} Suppose that $\om$ is an open subset of $\R^n$ and the measure of $\R^n\setminus \om$ is non-zero. Without loss of generality, suppose that the unit ball $B_1\subset \om$. Let $u\in C^2 (\om)$ be a nonnegative solution of
\be\label{eq:main}
(-\Delta_\om)^\sigma u= u^{\frac{n+2\sigma}{n-2\sigma}} \quad \mbox{in }B_1.
\ee
If $n\ge 4\sigma$, then
\[
\|u\|_{C^2(B_{1/2})}\le C(n, \sigma,\om),
\]
where $C(n, \sigma,\om)>0$ is a constant depending only $n,\sigma, \om $.
\end{thm}

Theorem \ref{thm:main} is of nonlocal nature and fails when $\sigma=1$. Since no condition is assumed on solutions in the complement of $B_1$, there exist infinitely many solutions of \eqref{eq:main}. Note that \eqref{eq:main} is the Euler-Lagrange equation of the fractional Sobolev inequality in $\om$. Recently,  Frank, Jin and Xiong \cite{FJX} showed that the best constants of fractional Sobolev inequality depend on domains and can be achieved in many cases, which is different from the classical Sobolev inequalities in domains.

For every smooth bounded function  $u$ defined in $\om$, by extending $u$ to zero outside $\om$ we see that
\be \label{eq:intro-2}
(-\Delta_\om)^\sigma u(x)=(-\Delta)^\sigma u(x)-A_\om(x) u(x) \quad \mbox{for }x\in \om,
\ee  where $(-\Delta)^\sigma:=(-\Delta_{\R^n})^\sigma$ is the standard fractional Laplacian,
\be
A_\om(x):= c(n,\sigma)\int_{\R^n\setminus \om} \frac{1}{|x-y|^{n+2\sigma}}\,\ud y.
\ee
Since the measure of $\R^n\setminus \om$ is positive, $A_\om>0$. Generally, let us consider the equation
\be \label{eq:main1}
(-\Delta)^\sigma u- a(x) u=u^{\frac{n+2\sigma}{n-2\sigma}} \quad \mbox{in }B_3, \quad u\ge 0\quad  \mbox{in }\R^n,
\ee
where the potential $a(x)$ is assumed to be nonnegative and smooth.

Second order critical semilinear elliptic equations of \eqref{eq:main1} type have been studied very extensively. A typical example is the Yamabe equation on Riemannian manifolds whose potential is the scalar curvature multiplied by a constant. Compactness and blow-up phenomenon of solutions to the Yamabe equation have been well understood; see, e.g., the recent book Hebey \cite{Hebey} and references therein. Note that the Laplacian of scalar curvature at the center of conformal normal coordinates equals $-\frac16 |W_g|^2$, where $W_g$ is the Weyl tensor of the metric $g$.  A conjecture due to Schoen says if there exists a sequence of local solutions to the Yamabe equation that blow up at $x_i\to \bar x$ then the Weyl tensor will vanish at $\bar x$ up to $[\frac{n-6}{2}]$-th order derivatives, where $n$ is the dimension of manifolds. If $6\le n\le 24$, the conjecture was proved positively by Li-Zhang \cite{Li-Zhang05,Li-Zhang06}, Marques \cite{Marques} and Khuri-Marques-Schoen \cite{KMS}. If $n\ge 25$, a counterexample was obtained by Marques \cite{Marques-2}. Consequently, solutions set of the Yamabe equation is compact in $C^2$ if the Weyl tensor or some derivatives of order $\le [\frac{n-6}{2}]$ does not vanish everywhere in dimension less than $24$. If the Weyl tensor does not vanish everywhere, compactness was proved in all dimensions $n\ge 6$ by \cite{Li-Zhang05,Marques}. Similar phenomenon has been proved recently by Li-Xiong \cite{Li-Xiong} for the fourth order $Q$-curvature equation in dimension $n\ge 8$. Another purpose of paper is to establish an analogue for Yamabe type equations with non-geometric potentials.

Let us introduce the space
\[
\mathcal{L}_\sigma(\R^n)=\{u\in L^1_{loc}(\R^n): \int_{\R^n}\frac{|u(x)|}{(1+|x|)^{n+2\sigma}}\,\ud x<\infty\}.
\]
Even though the two theorems below are stated in the nonlocal setting, they can be extended to $\sigma=1$.

\begin{thm}\label{thm:2} Let $ u\in C^2(B_3)\cap \mathcal{L}_\sigma(\R^n)$ be  a solution of \eqref{eq:main1} with $a\ge 0$ and $n\ge 4\sigma$. If either
\begin{itemize}
\item[(i)] $a>0$ in $B_2$, or
\item[(ii)] $\Delta a>0 $  on $\{x:a(x)=0\}\cap B_2$ and $n\ge 4\sigma+2$
\end{itemize} holds, then
\[
\|u\|_{C^2(B_{1})}\le C,
\]
where $C>0$ depends only on $n, \sigma, \|a\|_{C^4(B_3)}$ and $\inf_{B_2} a$  if (i) holds, otherwise it depends only on $n, \sigma, \|a\|_{C^4(B_3)}$ and $\inf_{\{x:a(x)=0\}\cap B_2}\Delta a$.

\end{thm}

In view of \eqref{eq:intro-2}, Theorem \ref{thm:main} follows from Theorem \ref{thm:2}. We believe there are blow-up examples if $n<4\sigma$ and $a>0$.   Compactness of finite energy changing-signs solutions of Brezis-Nirenberg problem was established in dimensions $n>6\sigma$ by Devillanova-Solimini \cite{DS} for $\sigma=1$ and Yan-Yang-Yu \cite{YYY} for $0<\sigma<1$, where $a$ is a positive constant.
Corresponding to \cite{Li-Zhang05,Marques}, we have:

\begin{thm}\label{thm:3} Let $ u_i\in C^2(B_3)\cap \mathcal{L}_\sigma(\R^n)$, $i=1,2,\dots$, be a solution of
\be\label{eq:main2}
(-\Delta)^\sigma u_i- a_i(x) u_i=u_i^{\frac{n+2\sigma}{n-2\sigma}} \quad \mbox{in }B_3, \quad u_i\ge 0\quad  \mbox{in }\R^n,
\ee
where $a_i\ge 0$, $\|a_i\|_{C^4(B_3)}\le A_0$ for some $A_0>0$ and $a_i\to a$ in $C^4(B_3)$. Suppose that $\Delta a_i\ge 0 $  in $\{x:a_i(x)< \va\}\cap B_2$ for some $\va>0$ independent of $i$ and $n\ge 4\sigma+2$. If  $x_i\to \bar x\in B_1$ and $u_i(x_i)\to \infty$ as $i\to \infty$, then $a(\bar x)=\Delta a(\bar x)=0$. Furthermore,
\begin{itemize}
\item[(i)] If $4\sigma+2\le n<6\sigma+2$, we can find $x'_i\to \bar x$ such that
\[
a_i(x_i') u(x_i')^{\frac{4}{n-2\sigma}}\ln u_i(x_i')+ \Delta a_i(x_i') \le C  (\ln u_i(x_i'))^{-1}
\]
for $n=4\sigma+2$ and
\[
a_i(x_i') u(x_i')^{\frac{4}{n-2\sigma}}+ \Delta a_i(x_i') \le C  u(x_i')^{\frac{2(4\sigma+2-n)}{n-2\sigma}}
\]
for $4\sigma+2< n<6\sigma+2$, where $C>0$ depends only on $n,\sigma$, $\va$ and $A_0$.

\item[(ii)] If $n\ge 6\sigma+2$, assume that
\be \label{eq:main-assum}
x_i \mbox{ is a local maimum point of }u_i, \quad  \max_{B_{\bar d}(x_i)}u_i(x)\le \bar b u_i(x_i)
\ee
for some positive constants $\bar b$ and $\bar d $. Then
\[
a_i(x_i) u(x_i)^{\frac{4}{n-2\sigma}}+ \Delta a_i(x_i) \le C \begin{cases} u_i(x_i)^{\frac{-4\sigma}{n-2\sigma}} \ln u_i(x_i)& \quad \mbox{for }n=6\sigma+2,\\
u_i(x_i)^{\frac{-4\sigma}{n-2\sigma}} & \quad \mbox{for }n>6\sigma+2,
\end{cases}
\]
 where $C>0$ depends only on $n,\sigma$,  $\va$, $A_0$, as well as constants $\bar b$ and $\bar d$.
\end{itemize}

\end{thm}

It is interesting to point out that if $\sigma\ge 1$, the constant  $6\sigma+2$ would be replaced by $4\sigma+4$ and $u_i(x_i')^{\frac{4\sigma}{n-2\sigma}} $ by $u_i(x_i')^{\frac{4}{n-2\sigma}} $, see the proof Proposition \ref{prop:sign}. The borderlines of dimensions in the above theorem might be not applicable to the compactness problem of the fractional Yamabe equations on the conformal boundaries of Einstein-Poincar\'e manifolds, as the second order operators have a non-trivial zero-order term. Fractional conformal invariant operators and fractional Yamabe problem have been studied by Graham-Zworski \cite{GZ}, Chang-Gonz\'alez \cite{CG}, Case-Chang \cite{CC}, Gonz\'alez-Qing \cite{GQ}, Fang-Gonz\'alez \cite{FG} and Kim-Musso-Wei \cite{KMW2} recently. Non-compactness examples of the fractional Yamabe equations were obtained by Kim-Musso-Wei \cite{KMW1} in higher dimensions as Brendle \cite{B} and Brendle-Marques \cite{BM} did for the Yamabe equation.  The $1/2$-Yamabe problem coincides with the boundary Yamabe problem initiated by Escobar \cite{E}.  The compactness problem of $1/2$-Yamabe equation has been studied by Felli-Ould Ahmedou \cite{FO} and Almaraz \cite{de,de1}.

The proofs of main theorems rely on asymptotic analysis of blowing up solutions. First, we should understand the possible bubbles interaction caused by the non-locality. By now two methods have been developed:
\begin{itemize}
\item[1.] Using the extension formula of Caffarelli-Silvestre \cite{CaS}, see Jin-Li-Xiong \cite{JLX};
\item[2.] Using Green's representation, see Jin-Li-Xiong \cite{JLX1} and Li-Xiong \cite{Li-Xiong}.
\end{itemize}
We will use the first one in the paper. Except the interest of degenerate elliptic equations, it appears easier to be adapted to study fractional Yamabe equations mentioned above. In addition, our proofs of Theorem \ref{thm:2} and Theorem \ref{thm:3} imply that both of them are still true when $(-\Delta)^\sigma$ is replaced by the spectral fractional Laplace operator. See Cabr\'e-Tan \cite{CT}, Capella-D\'avila-Dupaigne-Sire \cite{CDDS}, Yan-Yang-Yu \cite{YYY} and many others for study of nonlinear problems involving spectral fractional Laplace operator.  The second method has prominent advantage in dealing with higher order elliptic equations.

By using Caffarelli-Silvestre extension, the blow up analysis procedure will need a B\^ocher type theorem for degenerate elliptic equations with isolated singularities. Existence of Green function of this type degenerate elliptic equations on manifolds were obtained by Jin-Xiong \cite{JX} and Kim-Musso-Wei \cite{KMW2} via a duality argument but asymptotic expansion seems unknown. A difficulty is the lack of weighted $W^{1,p}$ estimates.
In section \ref{sect:3}, we establish existence and asymptotic expansion of Green functions via \textit{parametrix method} with the help of half-space Riesz potentials. Our approach also works for degenerate elliptic equations on manifolds.

The proofs of Theorem \ref{thm:2} and \ref{thm:3} need a refined quantitative asymptotic analysis of that in Jin-Li-Xiong \cite{JLX}. In the second order case, such type analysis was developed first by Chen-Lin  \cite{Chen-Lin} for the prescribing scalar curvature equation and then references cited above for the Yamabe equation. Since potentials in \eqref{eq:main1} and \eqref{eq:main2} are not geometric and their Taylor expansion polynomials of order $\ge 2$ have not to be orthogonal to the zeroth and first order polynomials, it is not possible to construct correctors.  It is unclear to us how to show higher order derivatives vanishing estimates. Furthermore, we lose the algebraic structure used by Khuri-Marques-Schoen \cite{KMS} to construct correctors in polynomial form.

The organization of the paper is as follows. In section \ref{sec:2}, we prove a localizing lemma in metric spaces by extending a result in \cite{Li-Xiong}. It allows us to localize bubbles interaction in bounded domains. In section \ref{sect:3}, we prove the existence and uniqueness of Green's functions as well as B\^ocher type theorem. In section \ref{sect:analysis},  we establish basic results of so-called isolated simple blow up points. Compared with the counterpart of \cite{JLX}, several new ingredients are introduced. In section \ref{sec:5}, we establish the refined  quantitative asymptotic analysis mentioned above. In section \ref{sec:6}, we estimate the Pohozaev integral of blow up solutions. The main theorems are proved in section \ref{sec:7}.

\bigskip

\noindent\textbf{Acknowledgments:}
The authors are grateful to Professor YanYan Li for his patient guidance and constant encouragement.

\medskip

\section{A localizing lemma}
\label{sec:2}

In this section, we prove the following lemma, which extends a result in Li-Xiong \cite{Li-Xiong}.

\begin{lem}\label{lem:select} Let $(\mathcal{M},d)$ be a complete metric space. Let $S_i\subset \mathcal{M}$, $i=1,2\dots$, be a sequence of sets of finite points, however, the cardinality of $S_i$ may tend to infinity. Suppose that $x_i,y_i\in S_i$ are distinct points satisfying $x_i,y_i\to \bar x$ as $i\to \infty$.
Define $f_i:S_i\to (0,\infty)$ by
\[
f_i(x):=\min_{x'\in S_i\setminus \{x\}} d(x',x).
\]Let $R_i\to \bar R\in (1,\infty]$ satisfying $R_i f_i(x_{i})\to 0$. Then subject to a subsequence of $i\to \infty$ one
can find $z_i\in  S_i \cap \textbf{B}_{(2R_if_i(x_{i}) }( x_i)$ satisfying
\be \label{eq:s-1}
f_i(z_i) \le (2R_i+1)f_i(x_{i})
\ee
and
\be
\min_{x\in  S_i\cap \textbf{B}_{R_i f_i(z_i)} (z_i)} f_i(x) \ge \frac12 f_i(z_i),
\ee
where $\textbf{B}_{r}(x)=\{y\in \mathcal{M}: d(x,y)<r\}$ for $r>0$.

\end{lem}

\begin{proof} Suppose the contrary, then there exists $N\in \mathbb{N}$ such that for any $i\ge N$, $z_i$ in the lemma can not been selected. Since $f_i(x_i)\le (2R_i+1) f_i(x_i)$, by the contradiction hypothesis, there must exist  $x_{i,1}\in  S_i\cap \textbf{B}_{R_i f_i(x_{i})} (x_{i})$ such that $f_{i}(x_{i,1})<\frac12 f_{i}(x_{i})$. Denote $x_{i,0}=x_i$.  We can define $x_{i,l}\in S_i$, $l=1,2\dots$,  satisfying
$f_i(x_{i,l})<\frac{1}{2} f_i(x_{i,(l-1)})$
 and $0<d(x_{i,l},x_{i,(l-1)})<  R_i f_i(x_{i,(l-1)})$
inductively as follows. Once $x_{i,l}$, $l\ge 2
$, is defined, we have, for $2\le m\le l$, that
\[
d(x_{i,m},x_{i,(m-1)})<R_i f_i(x_{i,(m-1)})<R_i 2^{-1} f_i(x_{i,(m-2)})<\cdots <R_i 2^{1-m} f_i(x_{i}),
\]
which implies
$$
d(x_{i,l},x_{i})\le \sum_{m=1}^l d(x_{i,m},x_{i,(m-1)})
< R_i f_i(x_{i}) \sum_{m=1}^l 2^{1-m}
< 2R_i  f_i(x_{i}),
$$
and
$$
f_i(x_{i,l})\le d(x_{i,l}, x_{i})+f_i(x_{i})\le (2R_i+1) f_i(x_{i}).
$$
So $z_i:=x_{i,l}$ satisfies
 $z_i\in  S_i\cap \textbf{B}_{R_i f_i(x_i)}( x_i)$
and \eqref{eq:s-1}.
By  the contradiction hypothesis, there must exist  $x_{i,(l+1)}\in  S_i\cap \textbf{B}_{R_i f_i(x_{i,l})} (x_{i,l})$ such that $f_{i}(x_{i,(l+1)})<\frac12 f_{i}(x_{i,l})$.
But $S_i$ is a finite set and we can not work for all $l\ge 2$. Therefore, the lemma follows.

\end{proof}

\section{Green's function and B\^ocher type theorems}
\label{sect:3}

Hereby, we use capital letters, such as $X=(x,t)$, to denote points in $\R^{n+1}$, and $t\geq 0$ usually.
$\B_R(X)$ denotes as the ball in $\R^{n+1}$ with radius $R$ and center $X$, $\B^+_R(X)$ as $\B_R(X)\cap \R^{n+1}_+$, and $B_R(x)$ as the ball in $\R^{n}$ with radius $R$ and center $x$. We also write $\B_R(0), \B^+_R(0), B_R(0)$ as $\B_R, \B_R^+, B_R$ for short. We use $\pa' \B_R^+(X)=\pa \B_R^+(X) \cap \pa \R^{n+1}_+, \pa'' \B_{R}^+(X)=\pa \B_R^+(X) \setminus \pa' \B_R^+(X) $. Through the extension formulation for $(-\Delta)^\sigma$ in \cite{CaS}, the equation \eqref{eq:main1} is equivalent to a degenerate elliptic equation with a Neumann boundary condition in one dimension higher:
\be\label{eq:ex0-1}
\begin{cases}
\mathrm{div}(t^{1-2\sigma} \nabla_{X} U)=0 & \quad \mbox{in }\R^{n+1}_+,\\
\frac{\pa U}{\pa \nu^\sigma} =N(\sigma)a(x)u+ N(\sigma)u^{\frac{n+2\sigma}{n-2\sigma}} &\quad \mbox{for }x\in B_3,
\end{cases}
\ee
where  $N_\sigma=2^{1-2\sigma}\Gamma(1-\sigma)/\Gamma(\sigma)$,
$$
\frac{\pa U}{\pa \nu^\sigma}(x,0)= -\lim_{t\to 0^+} t^{1-2\sigma} \pa_t U(x,t),
$$
and $u(x)=U(x,0)$. Since the Dirichlet problem does not have uniqueness, the extension will always refer to the \textit{canonical} one obtained by Poisson type integral:
\be\label{eq:poisson}
U(x,t)=\mathcal{P}_\sigma * u(x,t)=\beta(n,\sigma) \int_{\R^n} \frac{t^{2\sigma}}{(|x-y|^2+t^2)^{\frac{n+2\sigma}{2}}} u(y)\,\ud y,
\ee
where $\beta(n,\sigma)$ is a normalization constant.

For every open set $\om \subset \R^{n+1}_+$, we denote $W^{1,p}(t^{1-2\sigma}, \om)$, $1\le p<\infty$, the weighted Sobolev space equipped with the norm
\[
\|U\|_{W^{1,p}(t^{1-2\sigma},\om)}= (\int_{\om} t^{1-2\sigma}(U^p+|\nabla U|^p)\,\ud x\ud t)^{\frac{1}{p}}.
\]
It is easy to check that if $u\in C^2(B_3)\cap \mathcal{L}_\sigma(\R^n)$, then $\mathcal{P}_\sigma *u\in W^{1,2}(t^{1-2\sigma}, B_\rho \times T)$ for any $\rho<3$ and $T>0$.
The weighted space $W^{1,2}(t^{1-2\sigma},\om)$ and  weak solutions in the space for linear equation
\[
\mathrm{div}(t^{1-2\sigma} \nabla U)=0 \quad \mbox{in } \mathcal{B}_1^+, \quad \frac{\pa }{\pa \nu^\sigma}U(x,0)=a(x)U(x,0)+b(x)
\]
can be found in Cabr\'e-Sire \cite{CSi}, Jin-Li-Xiong \cite{JLX} and etc.
Classical regularity theory, such as Harnack inequality, H\"older estimates and Schauder estimates still hold. However, there is no weighted $W^{1,p}$, $p>2$,  theory.

The Harnack inequality will be used repeatedly, and thus we state it here. One can find proofs from \cite{CSi} or \cite{TX}.
\begin{prop}\label{prop:harnack} Let $ U \in W^{1,2}(t^{1-2\sigma}, \B_R^+)$ be a nonnegative weak solution of
\[
\begin{cases}
\mathrm{div}(t^{1-2\sigma} \nabla_{X} U)=0 & \quad \mbox{in }\B_R^+,\\
\frac{\pa U}{\pa \nu^\sigma} = a(x) U(x,0) &\quad \mbox{on }\pa' \B_R.
\end{cases}
\]
 If $a\in L^p(B_R)$ for some $p>n/2\sigma$, then we have
\[
\sup_{\overline \B_{R/2}^+} U\leq C(R) \inf_{\overline \B_{R/2}^+} U,
\]
where $C$ depends only on $n,\sigma, R$ and $\|a\|_{L^p(B_{R})}$.
\end{prop}

Denote
\be\label{eq:flat-neumann}
\mathcal{N}_\sigma(x,t):=c(n,\sigma) |X|^{2\sigma-n},
\ee where $c(n,\sigma)$ is a normalization constant. Then
\begin{equation*}
\begin{cases}
\mathrm{div}(t^{1-2\sigma}\nabla \mathcal{N}_\sigma )=0& \quad \mbox{in }\R^{n+1}_+,\\
\frac{\pa }{\pa \nu^\sigma}\mathcal{N}_\sigma  =\delta_0& \quad \mbox{on } \pa\R^{n+1}_+,
\end{cases}
\end{equation*}
in distribution sense, where $\delta_0$ is the Dirac measure centered at $0$.

\begin{prop}\label{prop:greenexitence} Given a function $a\in L^\infty(B_1)$, we can find a constant  $0<\tau\le 1$, depending only on $n$, $\sigma$ and $\|a\|_{L^\infty(B_1)}$, such that there exists  $G(X)\in W^{1,2}(t^{1-2\sigma}, \mathcal{B}_{\tau}^+\setminus \mathcal{B}_{\rho}^+)$ for any $\rho>0$ satisfying
\begin{equation}\label{eq:toprove}
\begin{cases}
\mathrm{div}(t^{1-2\sigma}\nabla G)=0& \quad \mbox{in }\mathcal{B}_\tau^+,\\
\frac{\pa }{\pa \nu^\sigma} G=a G& \quad \mbox{on } \pa' \mathcal{B}^+_{\tau}\setminus\{0\},\\
G=0& \quad \mbox{on } \pa'' \mathcal{B}^+_{\tau},
\end{cases}
\end{equation}
in weak sense, and
\be\label{eq:asmp1}
\lim_{X\to 0}|X|^{n-2\sigma} G(X)=c(n,\sigma).
\ee
Here $c(n,\sigma)>0$ is the constant in \eqref{eq:flat-neumann}. Furthermore, if $a\in C^1(B_1)$, then
\be \label{eq:green's expan1}
G(X)=c(n,\sigma)|X|^{2\sigma-n}+E(X),
\ee
where $E(X)$ satisfies
\be\label{eq:greens'expan2}
|E(X)|+|X||\nabla_x E(X)|+|X|^{2\sigma} |t^{1-2\sigma}\pa _t E(X)|\le C|X|^{4\sigma-n}.
\ee

\end{prop}

\begin{proof}  Denote $V_0=c(n,\sigma)|X|^{2\sigma-n}$ and define inductively
\[
V_k(X)= \mathcal{N}_\sigma * (a V_{k-1})(x,t)=\int_{\R^n} \mathcal{N}_\sigma(x-y,t) a(y)V_{k-1}(y)\,\ud y \quad \mbox{if }|X|\le 2,
\]
and
\[
V_k(X)=0 \quad \mbox{if }|X|\ge 2, \quad k=1,2,\dots, [\frac{n}{2\sigma}].
\]
Clearly, $V_k\in C^\infty(\bar \B_{3/2}^+\setminus \{0\})$, for $k<[\frac{n}{2\sigma}]$ we have
\[
|V_{k}(X)|\le C |X|^{2\sigma(k+1)-n} \quad \forall~ |X|\le 1
\]
and $V_{[\frac{n}{2\sigma}]}(X)$ is H\"older continuous in $\B_1^+$. Furthermore,
\begin{equation*}
\begin{cases}
\mathrm{div}(t^{1-2\sigma}\nabla V_k)=0& \quad \mbox{in }\mathcal{B}_2^+,\\
\frac{\pa }{\pa \nu^\sigma} V_k=a V_{k-1}& \quad \mbox{on } \pa' \mathcal{B}^+_{2}\setminus\{0\}
\end{cases}
\end{equation*}
in weak sense.
Let
\[
V=\sum_{k=0}^{[\frac{n}{2\sigma}]} V_k.
\]

Choose $\tau$ to be small such that
\be\label{eq:coercive}
\frac{1}{2}\int_{\mathcal{B}_\tau^+} t^{1-2\sigma} |\nabla \varphi|^2- \int_{\pa' \mathcal{B}_\tau^+} a\varphi^2\ge 0 \quad  \forall ~\varphi \in W^{1,2}(t^{1-2\sigma},\mathcal{B}_\tau^+ ), \varphi =0 \mbox{ on }\pa'' \mathcal{B}_\tau^+.
\ee
By Lax-Milgram theorem,
\begin{equation}\label{eq:reduced}
\begin{cases}
\mathrm{div}(t^{1-2\sigma}\nabla W)=0& \quad \mbox{in }\mathcal{B}_\tau^+,\\
\frac{\pa }{\pa \nu^\sigma} W=a W+a V_{[\frac{n}{2\sigma}]}& \quad \mbox{on } \pa' \mathcal{B}^+_{\tau},\\
W=-V& \quad \mbox{on } \pa'' \mathcal{B}^+_{\tau}.
\end{cases}
\end{equation}
has a unique weak solution in $W^{1,2}(t^{1-2\sigma}, \mathcal{B}_\tau^+)$.
Let $G:=V+W$ and $E:= (V-V_0)+W$. By the construction of $V$ and the regularity theory in  \cite{JLX}, the proposition follows immediately.

\end{proof}

\begin{rem} \label{rem:mp} From the selecting $\tau$ by \eqref{eq:coercive}, the maximum principle holds in $B_{\tau}$. And thus $G(X)>0$ in $\mathcal{B}_\tau^+\cup \pa' \mathcal{B}_\tau^+$.
\end{rem}

\begin{prop}[B\^ocher type]\label{prop:bocher}
 Suppose that $U \in W^{1,2}(t^{1-2\sigma}, \mathcal{B}_{1}^+\setminus \mathcal{B}_{\rho}^+)$ for any $\rho>0$ is a nonnegative weak solution
\begin{equation}\label{eq:toprove-2}
\begin{cases}
\mathrm{div}(t^{1-2\sigma}\nabla U)=0& \quad \mbox{in }\mathcal{B}_1^+,\\
\frac{\pa }{\pa \nu^\sigma} U=a U& \quad \mbox{on } \pa' \mathcal{B}^+_{1}\setminus\{0\},\\
\end{cases}
\end{equation}
 where $a\in C^{1}(B_{1})$, then
\[
U(X)=A G(X)+H(X) \quad \mbox{for } 0<|X|\leq \tau,
\]
where $A$ is some nonnegative constant, $0<\tau<1$ and $G(X)$ are as in Proposition \ref{prop:greenexitence}, and $H$ is a $W^{1,2}(t^{1-2\sigma}, \mathcal{B}_\tau^+)$ weak solution of
\begin{equation*}
\begin{cases}
\mathrm{div}(t^{1-2\sigma}\nabla H)=0& \quad \mbox{in }\mathcal{B}_\tau^+,\\
\frac{\pa }{\pa \nu^\sigma} H=a H& \quad \mbox{on } \pa' \mathcal{B}^+_{\tau}.\\
\end{cases}
\end{equation*}
\end{prop}

By Proposition \ref{prop:blowupbubble}, we immediately have
\begin{cor} \label{cor:uniqueness} $G(X)$ constructed in Proposition \ref{prop:greenexitence} is unique.

\end{cor}

The proof of Proposition \ref{prop:bocher} adapts some idea from Li-Zhu \cite{Li-Zhu99}.

\begin{lem}\label{lem:bocher1} Assume the assumptions in Proposition \ref{prop:bocher}. If in addition
$$ U(X)=o(|X|^{2\sigma-n}) ~as~|X|\rightarrow 0,$$
 then $U\in W^{1,2}(t^{1-2\sigma}, \mathcal{B}_1^+)$ and thus $0$ is a removable singularity of $U$.
\end{lem}
\begin{proof} Let $\tau$ and $G$ be constructed in Proposition \ref{prop:greenexitence}. Let \begin{equation*}
\begin{cases}
\mathrm{div}(t^{1-2\sigma}\nabla \varphi )=0& \quad \mbox{in }\mathcal{B}_{\tau}^+,\\
\frac{\pa }{\pa \nu^\sigma} \varphi=a\varphi & \quad \mbox{on } \pa' \mathcal{B}_{\tau}^+,\\
\varphi= U& \quad \mbox{on }\pa'' \mathcal{B}_{\tau}^+.
\end{cases}
\end{equation*}
For any $\va>0$, let $\phi_\va=\va G+\varphi$. Since $U(X)=o(|X|^{2\sigma-n})$, one can find $\delta=\delta(\va)>0$ with $\lim\limits_{\va\to 0}\delta(\va)=0$ such that $\phi_\va>U$ on $\pa''\B_{\delta}$. By maximum principle (Remark \ref{rem:mp}), $U\le \phi_\va$ in $\B_\tau^+\setminus \B_{\delta}^+$. Sending $\va\to 0$, we have $U(X)\le \varphi(X)\le \max\limits_{\pa'' B_{\tau}^+} U$ for all $X\in \B_\tau^+\setminus \{0\}$.
Namely, $U\in L^\infty(\B_{1/2}^+)$.

Next, by the Schauder estimate for $U$, we have
\be\label{eq:refinebound}
|X||\nabla_x U(X)|+|X|^{2\sigma}|t^{1-2\sigma}\pa_t U|\le C.
\ee
For $\epsilon>0$, let $\eta_{\epsilon}$ be a cutoff function satisfying
\[\eta_{\epsilon}=
\begin{cases}
1\quad &\mbox{for} |X|\leq\epsilon,\\
0\quad  &\mbox{for} |X|\geq 2\epsilon,
\end{cases}
\quad |\nabla \eta_{\epsilon}|\leq\frac{C}{\epsilon},
\]
where $C>0$ depends only on $n$.  Using $U(1-\eta_\va)$ as a text function for the equation of $U$, we have
\[
0=-\int_{\mathcal{B}_1^+}t^{1-2\sigma}\nabla U\cdot\nabla\Big( U(1-\eta_{\epsilon})\Big)\,\ud X
+\int_{\partial'\mathcal{B}_1^+}a U^2(1-\eta_{\epsilon})\,\ud x.
\]
Since $U$ is bounded and \eqref{eq:refinebound}, we have
\begin{align*}
\int_{\mathcal{B}_1^+}t^{1-2\sigma}|\nabla U|^2(1-\eta_\va)\,\ud X &\le C+ \frac{C}{\va}\int_{\mathcal{B}_{2\va}^+\setminus \mathcal{B}_{\va}^+}t^{1-2\sigma}|\nabla U|\\&
\le C+ \frac{C}{\va}\va^{n+1-2\sigma}\le C.
\end{align*}
Sending $\va\to 0$, we have
\[
\int_{\mathcal{B}_1^+}t^{1-2\sigma}|\nabla U|^2\,\ud X\le C.
\]

In conclusion, we showed $U\in W^{1,2}(t^{1-2\sigma}, \mathcal{B}_1^+)\cap L^\infty( \mathcal{B}_1^+)$. The proposition follows immediately.

\end{proof}

From the proof of  Lemma \ref{lem:bocher1}, the condition $U\ge 0$ can be removed.

\begin{lem}\label{lemm:upasmp} Assume the assumptions in Proposition \ref{prop:bocher}. Then
$$A:=\overline{\lim\limits_{r\rightarrow0}}\max_{|X|=r}U(X)|X|^{n-2\sigma}<\infty.$$
\end{lem}
\begin{proof}
By Harnack inequality (Proposition \ref{prop:harnack}), for $0<r<1$ we have
$$\max_{\pa''\mathcal{B}_r^+}U(X)\leq C\min_{\pa''\mathcal{B}_r^+}U(X),$$
where $C(r)>0$ depends only on $n,\sigma$ and $\|a\|_{L^\infty(B_1)}$.
Let $\tau>0$ and $G(X)$ as in Proposition \ref{prop:greenexitence}. If $A=\infty$, using maximum principle (see Remark \ref{rem:mp}) we have
\[
U(X)\ge k G(X) \quad  \mbox{for all }k>0.
\]
This is impossible. Therefore, the lemma is proved.
\end{proof}

\begin{proof}[Proof of Proposition \ref{prop:bocher}] Let $\tau>0$ and $G(X)$ be as in Proposition \ref{prop:greenexitence}.
Set
$$
\bar A=\sup\{\lambda\geq 0| \lambda G(X)\leq U(X)~\forall~X\in \mathcal{B}_{\tau}^+\setminus\{0\}\}.
$$
It follows from Lemma \ref{lemm:upasmp} that $0\leq \bar A\leq A  <\infty$.

Case 1: $\bar A=0$.

We claim that  for any $\epsilon>0$, there exists $r_{\epsilon}\in(0,\tau)$ such that
$$ \min_{|X|=r}\{U(X)-\epsilon G(X)\}\leq0 \quad \forall~ 0< r<r_{\epsilon}.$$
If the above claim were false, then there would exist some $\epsilon_{0}>0$ and $r_{j}\rightarrow0^{+}$ such that
$$ \min_{|X|=r_{j}}\{U(X)-\epsilon_{0} G(X)\}>0.$$
Notice that $U(X)-\epsilon_{0} G(X)\geq 0 $ for $|X|=\tau$. We derive from the maximum principle that $U(X)-\epsilon_{0} G(X)\geq 0$ on $\mathcal{B}_{\tau}^+\setminus\mathcal{B}_{r_{j}}^+$. It follows that $U(X)-\epsilon_{0} G(X)\geq 0$ on $\mathcal{B}_{\tau}^+\setminus \{0\}$ which implies that $
\bar{A}\geq\epsilon_{0}>0$, a contradiction.\par
Therefore, for any $\epsilon>0$, and $0< r<r_{\epsilon}$, there exists $X_{\epsilon}$ with $|X_{\epsilon}|=r$ such that $U(X_{\epsilon})\leq\epsilon G(X_{\epsilon})$. By Harnack inequality, we have
$$ \max_{|X|=r}U(X)\leq CU(X_{\epsilon})\leq C\epsilon G(X_{\epsilon}).$$
It follows that
$$ U(X)=o(|X|^{2\sigma-n})~as~|X|\rightarrow0~.$$ By Lemma \ref{lem:bocher1}, the singularity is removable.

Case 2: $\bar A>0.$
We consider $H(X)=U(X)-\bar A G(X)$. From the definition of $\bar A$, we know that $H(X)\geq0$. By the maximum principle, we know that either $H(X)=0$  or $H(X)>0$ in $\mathcal{B}_{\tau}^+\setminus \{0\}$. In the former case we are done.

In the latter case, $H(X)$ satisfies \eqref{eq:toprove-2} with $\mathcal{B}_1^+$ replaced by $\mathcal{B}_\tau^+$.
Set
$$ b=\sup\{\lambda\geq 0| \lambda G(X)\leq H(X)~\forall~X\in \mathcal{B}_{\tau}^+\setminus\{0\}\}.$$
Arguing as in case 1, wee have $b=0$ and $H(X)=o(|X|^{2\sigma-n})$. By Lemma \ref{lem:bocher1}, $H\in W^{1,2}(t^{1-2\sigma}, \mathcal{B}_{\tau}^+)$. We are done again.

Therefore, the proposition is proved.
\end{proof}

\section{Analysis of isolated blow up points}
\label{sect:analysis}

In this section, we follow Jin-Li-Xiong \cite{JLX}, but several new ingredients are needed to deal with the linear term. For example, a conformal type transform will be used to show the sharp upper bound of blow up solutions; see Lemma \ref{lem:toup}.

Let $\tau_i\geq 0$ satisfy
$\lim\limits_{i\rightarrow \infty}\tau_i=0$, $p_i=(n+2\sigma)/(n-2\sigma)-\tau_i$, and $a_i\ge 0$ be a sequence of functions converging to $a$ in $C^{2}(B_3)$, and $\{u_{i}\}$ be a sequence of  $C^2(B_3)\cap \mathcal{L}_\sigma(\R^n)$ solutions of
\be\label{eq:mainseq}
(-\Delta)^{\sigma}u_i=a_{i}(x)u_{i}+u_i^{p_i} \quad \mbox{in } B_3, \quad u_i\ge 0 \quad \mbox{in }\R^n.
\ee
Let $U_i=\mathcal P_{\sigma}*u_{i}$ be the extension of $u_i$ as in \eqref{eq:poisson}. Then we have
\be\label{eq:mainseq-ext}
\begin{cases}
\mathrm{div}(t^{1-2\sigma}\nabla U_i)=0,&\quad \mbox{in } \R^{n+1}_+,\\
\frac{\pa U_i(x,0) }{\pa \nu^\sigma}=a_i(x)U_{i}(x,0)+U_i(x,0)^{p_i},&\quad  x\in B_3,
\end{cases}
\ee
where we dropped the harmless constant $N(\sigma)$ for brevity.

A point $\bar{y}\in B_2$ is called a blowup point of $\{u_{i}\}$ if $u_{i}(y_{i})\rightarrow\infty$ for some $y_{i}\rightarrow\bar{y}$.

\begin{defn}\label{def:iso}
Let $\{u_i\}$ satisfy \eqref{eq:mainseq}. We say a point $\bar  y\in B_2$ is an isolated blow up point of $\{u_i\}$ if there exist
$0<\bar  r<\mbox{dist}(\bar  y,\partial B_3)$, a constant $\tilde{C} >0$, and a sequence $y_i$ tending to $\bar  y$, such that,
$y_i$ is a local maximum of $u_i$, $u_i(y_i)\rightarrow \infty$ and
\[
u_i(y)\leq \tilde{C} |y-y_i|^{-2\sigma/(p_i-1)} \quad \mbox{for all } y\in B_{\bar  r}(y_i).
\]
\end{defn}

Let $y_i\rightarrow \bar  y$ be an isolated blow up point of $u_i$, define for $0< r<\bar{r}$,
\be\label{def:average}
\bar  u_i(r)=\frac{1}{|\pa B_r(y_{i})|} \int_{\pa B_r(y_i)}u_i\quad \mbox{and} \quad \bar  w_i(r)=r^{2\sigma/(p_i-1)}\bar  u_i(r).
\ee

\begin{defn}\label{def:isosim}
We say $y_i \to \bar  y\in B_2$ is an isolated simple blow up point, if $y_i \to \bar  y$ is an isolated blow up point, such that, for some
$\rho>0$ (independent of $i$) $\bar  w_i$ has precisely one critical point in $(0,\rho)$ for large $i$.
\end{defn}

In the above, we use $B_2$ and $B_3$ for conveniences. One can replace them by open sets.

\begin{lem}\label{lem:s-harnack} Suppose that $u_i$
is a sequence of solutions of \eqref{eq:mainseq}, and $y_i\to 0$ is an isolated blow up point of $\{u_i\}$, i.e., for some positive constants $A_1$ and $\bar r$ independent of $i$,
\be\label{eq:iso}
|y-y_i|^{2\sigma/(p_i-1)}u_i(y)\leq A_1,\quad \mbox{for all } y\in B_{\bar r}(y_{i})\subset B_3.
\ee
Then for any $0<r<\frac13 \overline r$, we have the following Harnack inequality
\[
\sup_{\mathcal{B}^+_{2r}(Y_i)\setminus\overline{\mathcal{B}^+_{r/2}(Y_i)}} U_i\leq C \inf_{\mathcal{B}^+_{2r}(Y_i)\setminus\overline{\mathcal{B}^+_{r/2}(Y_i)}} U_i,
\]
where $Y_i=(y_i,0)$ and $C>0$ depends only on $n, \sigma, A_1, \bar r$ and $\dsup_i\|a_i\|_{L^\infty(B_{\overline r}(y_i))}$.
\end{lem}

\begin{proof} It follows from applying  Proposition \ref{prop:harnack} to $ r^{\frac{2\sigma}{p_i-1}} U_i(r X+Y_i)$. See the proof of Lemma 4.3 of \cite{JLX} for more details.
\end{proof}

Without loss of generality, we assume $\bar r=2$ to the end of the section.

\begin{prop}\label{prop:blowupbubble} Assume as  in Lemma \ref{lem:s-harnack}. Suppose that $\|a_i\|_{C^{2}(B_3)}\le A_0$.
Then for any $R_i\rightarrow \infty$, $\va_i\rightarrow 0^+$, we have,
after passing to a subsequence (still denoted as $\{u_i\}$,
$\{y_i\}$, etc. ...), that

\begin{equation}\label{limit1}
\|m_i^{-1}u_i(m_i^{-(p_i-1)/2\sigma}\cdot+y_i)-\bar{c}(1+|\cdot|^2)^{(2\sigma-n)/2}\|_{C^2(B_{2R_i}(0))}\leq \va_i,
\end{equation}

\be\label{limit2}
R_i m_{i}^{-\frac{p_i-1}{2\sigma}}\rightarrow 0\quad \mbox{as}\quad i\rightarrow \infty,
\ee
where $m_i=u_i(y_i)$ and $\bar c$ depends only on $n$ and $\sigma$.
\end{prop}

\begin{proof}
See the proof of Proposition 4.4 of \cite{JLX}.
\end{proof}

In the sequel, we will always work on the sequences $R_i\to \infty$ and $\va_i\to 0$ which ensure \eqref{limit1} and \eqref{limit2} valid.

\begin{prop}\label{prop:low-1} Under the hypotheses of Proposition \ref{prop:blowupbubble}, there exists a positive constant $C=C(n,\sigma, A_0, A_1)$ such that,
\[
u_i(y)\geq C^{-1}m_i(1+\bar c m_i^{(p_i-1)/\sigma}|y-y_i|^2)^{(2\sigma-n)/2}, \quad |y-y_i|\leq 1.
\]
In particular, for any $e\in \mathbb{R}^n$, $|e|=1$, we have
\[
u_i(y_i+e)\geq C^{-1}m_i^{-1+((n-2\sigma)/2\sigma)\tau_i},
\]
where $\tau_i=(n+2\sigma)/(n-2\sigma)-p_i$.
\end{prop}

\begin{proof} Since $a_i\ge 0$, the proof is the same as that of Proposition 4.5 of \cite{JLX}.
\end{proof}

\begin{lem}\label{lem:up-1} In addition to the hypotheses of Proposition \ref{prop:blowupbubble}, suppose further that $y_i\to 0$ is an isolated simple blow up point of $\{u_i\}$ with a constant $\rho>0$. Assume $R_{i}\rightarrow\infty$ and $\va_i\rightarrow 0^{+}$ are sequences with which \eqref{limit1} and \eqref{limit2} hold. Then for any $0<\delta<<(n-2\sigma)/2$, we have
\[
u_i(y)\leq C u_i(y_i)^{-\lda_i}|y-y_i|^{2\sigma-n+\delta},\quad \mbox{for all }r_i\leq |y-y_i|\leq 1,
\]
where $\lda_i=(n-2\sigma-\delta)(p_i-1)/2\sigma-1$ and $C>0$ depends only on $n,\sigma, A_0, A_1$ and $\delta$.
\end{lem}

\begin{proof}  The proof is similar to that of Lemma 4.6 of \cite{JLX}, but here $\delta>0$ can not be a sequence $\delta_i\to 0$ as in \cite{JLX}.
From Proposition \ref{prop:blowupbubble}, we see that
\be\label{4.8}
u_i(y)\leq C u_i(y_i)R_i^{2\sigma-n} \quad \mbox{for all } |y-y_i|=r_i.
\ee
Let $\overline u_i(r)$ be the average of $u_i$ over the sphere of radius $r$ centered at $y_i$.
It follows from the assumption of isolated simple
blow up and Proposition \ref{prop:blowupbubble} that
\be\label{4.9}
r^{2\sigma/(p_i-1)}\overline u_i(r) \quad \mbox{is strictly decreasing for $r_i<r<\rho$}.
\ee
By Lemma  \ref{lem:s-harnack},  \eqref{4.9} and \eqref{4.8},  we have, for all $r_i<|y-y_i|<\rho$,
\[
\begin{split}
|y-y_i|^{2\sigma/(p_i-1)}u_i(y)&\leq C|y-y_i|^{2\sigma/(p_i-1)}\overline u_i(|y-y_i|)\\&
\leq r_i^{2\sigma/(p_i-1)}\overline u_i(r_i)
\leq CR_i^{\frac{2\sigma-n}{2}+o(1)},
\end{split}
\]
where $o(1)$ denotes some quantity tending to $0$ as $i\to \infty$.
Applying Lemma  \ref{lem:s-harnack} again, we obtain
\be\label{4.10}
U_i(Y)^{p_i-1}\leq O(R_i^{-2\sigma+o(1)})|Y-Y_i|^{-2\sigma}\quad \mbox{for all }  r_i\leq |Y-Y_i|<\rho.
\ee
Consider operators
\[
\begin{cases}\mathfrak{L}(\Phi)=\mathrm{div}(s^{1-2\sigma}\nabla \Phi(Y)) \quad &\mbox{in } \mathcal{B}^+_2,\\
L_i(\Phi)=\frac{\pa }{\pa \nu^\sigma} \Phi(y,0)-[a_i(y)+ u_i^{p_i-1}(y)]\Phi(y,0)\quad &\mbox{on } \pa'\mathcal{B}^+_2
\end{cases}
\]
for $\Phi\in W^{1,2}(t^{1-2\sigma}, \B_2^+)$.
Clearly, $U_i>0$ satisfies $\mathfrak{L}(U_i)=0$ in $\mathcal{B}^+_2$ and $L_i(U_i)=0$ on $\pa'\mathcal{B}^+_2$.\par~\par
\par
For $0\leq \mu\leq n-2\sigma$ and $\va>0$, a direct computation yields
\[
\begin{split}
&\mathfrak{L}(|Y-Y_i|^{-\mu}-\va s^{2\sigma}|Y-Y_i|^{-(\mu+2\sigma)})\\&
=s^{1-2\sigma}|Y-Y_i|^{-(\mu+2)}\Big\{-\mu(n-2\sigma-\mu)
+\frac{\va (\mu+2\sigma)(n-\mu) s^{2\sigma}}{|Y-Y_i|^{2\sigma}} \Big\}
\end{split}
\]
and
\[
\begin{split}
&L_i(|Y-Y_i|^{-\mu}-\va s^{2\sigma}|Y-Y_i|^{-(\mu+2\sigma)})\\&
=\Big\{2\varepsilon \sigma-  (a_i(y)+ u_i^{p_i-1}(y))|Y-Y_i|^{2\sigma}\Big\}|Y-Y_i|^{-(\mu+2\sigma)}.
\end{split}
\]
Hence, for fixed $\delta>0$, we can choose $\va>0$ small such that for
$r_i\leq|Y-Y_i|<\rho$,
\[
\begin{split}
&\mathfrak{L}(|Y-Y_i|^{-\delta}-\va s^{2\sigma}|Y-Y_i|^{-(\delta+2\sigma)})\leq 0,\\
&\mathfrak{L}(|Y-Y_i|^{-(n-2\sigma-\delta)}-\va s^{2\sigma}|Y-Y_i|^{-(n-\delta)})\leq 0.
\end{split}
\]
Now $\va$ is fixed. Then we can find $0<\rho_1\le \rho$, depending only on $n,\sigma, A_0,A_1$ and $\va$, such that for $r_i\leq|y-y_i|<\rho_1$,
\[
\begin{split}
&L_i(|Y-Y_i|^{-\delta}-\va s^{2\sigma}|Y-Y_i|^{-(\delta+2\sigma)})
\geq 0,\\
&L_i(|Y-Y_i|^{-(n-2\sigma-\delta)}-\va s^{2\sigma}|Y-Y_i|^{-(n-\delta)})\geq 0.
\end{split}
\]

Set $M_i=\max_{\pa'' \mathcal{B}^+_{\rho_1}} U_i$, $\lda_i=(n-2\sigma-\delta)(p_i-1)/2\sigma-1$ and
\[
\begin{split}
\Phi_i(Y)=&2M_i \rho_{1}^{\delta}(|Y-Y_i|^{-\delta}-\va s^{2\sigma}|Y-Y_i|^{-(\delta+2\sigma)})\\
&+2Au_i(y_i)^{-\lda_i}(|Y-Y_i|^{2\sigma-n+\delta}-\va s^{2\sigma}|Y-Y_i|^{-n+\delta}),
\end{split}
\]
where $A>1$ will be chosen later. By the choice of $M_i$ and  $\lda_i$, we immediately have
\[
\Phi_i(Y)\geq M_i\geq U_i(Y) \quad \mbox{for all } |Y-Y_i|=\rho_1.
\]
\[
\Phi_i\geq AU_i(Y_i)R_i^{2\sigma-n+\delta}\geq AU_i(Y_i)R_i^{2\sigma-n}\quad \mbox{for all }|Y-Y_i|=r_i.
\]
Due to (\ref{4.10}), we can choose $A$ to be sufficiently large such that
\[
\Phi_i\geq U_i\quad \mbox{for all }|Y-Y_i|=r_i.
\]
Applying the maximum principle in Lemma A.3 of \cite{JLX} to $\Phi_i-U_i$ in $\B_{\rho_1}^+\setminus  \bar \B_{r_i}^+$, it yields
\be \label{eq:up-b}
U_i\leq \Phi_i\quad \mbox{for all }r_i\leq |Y-Y_i|\leq \rho_1.
\ee
For $r_i<\theta<\rho_1$, by \eqref{4.9} and Lemma \ref{lem:s-harnack} we have
\[
\begin{split}
\rho_1^{2\sigma/(p_i-1)}M_i&\leq C \rho_1^{2\sigma/(p_i-1)}\bar u_i(\rho_1)\\
&\leq C\theta^{2\sigma/(p_i-1)}\bar u_i(\theta)\\
&\leq C\theta^{2\sigma/(p_i-1)}\{M_i \rho_{1}^{\delta}\theta^{-\delta}+Au_i(y_i)^{-\lda_i}\theta^{2\sigma-n+\delta}\}.
\end{split}
\]
Choose $\theta=\theta(n,\sigma,\rho,A_0,A_1)$ sufficiently small so that
\[
C\theta^{2\sigma/(p_i-1)}\rho_1^{\delta}\theta^{-\delta}\leq \frac12 \rho_1^{2\sigma/(p_i-1)}.
\]
It follows that
\[
M_i\leq Cu_i(y_i)^{-\lda_i}.
\]
Together with \eqref{eq:up-b}, Lemma \ref{lem:up-1} holds when $|y-y_i|\le \rho_1$. By Lemma \ref{lem:s-harnack} it also holds when $\rho_1\le |y-y_i|\le 1$.

Therefore, we complete the proof.
\end{proof}

\begin{prop}[Pohozaev type]\label{prop:pohozaev} Let $U\in W^{1,2}(t^{1-2\sigma},\B_{2R}^+)$ and $U\geq 0$ in $\B_{2R}^+$ be a weak solution of
\be\label{poh}
\begin{cases}
\mathrm{div}(t^{1-2\sigma}\nabla U)=0&\quad \mbox{in }\mathcal{B}_{2R}^+,\\
\frac{\pa}{\pa \nu^\sigma} U(x,0)=a(x)U(x,0)+U^{p}(x,0)&\quad \mbox{on }\pa'\B_{2R}^+,
\end{cases}
\ee
where $a\in C^1(B_{2})$ and $p>0$. Then
\be\label{eq:poh-id}
 P_\sigma(0,R,U)+Q_\sigma(0,R,U,p)=0,
\ee
where
\[
P_\sigma(0,R,U):=\int_{\pa'' \mathcal{B}^+_R}t^{1-2\sigma} \left(\frac{n-2\sigma}{2}U\frac{\pa U}{\pa \nu}-\frac{R}{2}|\nabla U|^2+R|\frac{\pa U}{\pa \nu}|^2\right)\,\ud S,
\]
\[
\begin{split}
Q_\sigma(0,R,U,p):=&(\frac{n-2\sigma}{2}-\frac{n}{p+1})\int_{B_R} U(x,0)^{p+1}\,\ud x\\&- \int_{B_R} (\sigma a(x)+\frac12 x\nabla a(x))U(x,0)^2\,\ud x +R\int_{\pa B_R} \frac{1}{2} aU^2+\frac{1}{p+1}U^{p+1}\,\ud S
\end{split}
\]
and $\nu$ is the unit out normal to $\pa \B_R$.
\end{prop}
\begin{proof}   See the proof of Proposition 4.7 of \cite{JLX}.

\end{proof}

\begin{lem}\label{lem:tau-1} Assume as in Lemma \ref{lem:up-1}. Choose $\delta$ small,  then
\[\tau_{i}=
O(u_{i}(y_{i})^{-\min\{\frac{2}{n-2\sigma},1\}})
\]
Consequently,
\[u_{i}(y_{i})^{\tau_{i}}\rightarrow 1.\]
\end{lem}

\begin{proof} Denote $Y_i=(y_i,0)$. By Proposition \ref{prop:pohozaev} for equation \eqref{eq:mainseq-ext}, we have
\be\label{4.11}
\begin{split}
 \frac{(n-2\sigma)\tau_i}{2(p_i+1)}\int_{B_\rho(y_i)}U_i^{p_i+1}= &-\int_{B_{\rho}(y_i)}(\sigma a_i+\frac{1}{2}(y-y_i)\nabla a_i)U_i(y,0)^2 \\& +
  \rho\int_{\pa B_\rho(y_i)} \frac{1}{2} aU^2+\frac{1}{p+1}U^{p_i+1}\,\ud S+ P_\sigma(Y_i,\rho ,U_i),
\end{split}
\ee
It follows from Proposition \ref{prop:low-1} that
\be\label{4.12}
\begin{split}
\int_{B_\rho(y_i)}U_i^{p_i+1}&
\geq C^{-1}\int_{B_\rho(y_i)}\frac{m_i^{p_i+1}}{(1+\bar{c}|m_i^{(p_i-1)/2\sigma}(y-y_i)|^2)^{(n-2\sigma)(p_i+1)/2}}\\&
\geq C^{-1} m_i^{\tau_i(n/2\sigma-1)}\int_{B_{\rho m_i^{(p_i-1)/2\sigma}}}\frac{1}{(1+\bar{c}|z|^2)^{(n-2\sigma)(p_i+1)/2}}\\&
\geq C^{-1}m_i^{\tau_i(n/2\sigma-1)},
\end{split}
\ee
where we used change of variables $z=m_i^{(p_i-1)/2\sigma}(y-y_i)$ in the second inequality.
By Lemma \ref{lem:up-1}, we have
\be\label{4.1444}
\begin{split}
&\int_{B_\rho(y_i)\setminus B_{r_i}(y_i)}|y-y_i|u_i^{2}\\&
\leq\int_{B_\rho(y_i)\setminus B_{r_i}(y_i)}|y-y_i|(u_{i}(y_{i})^{-\lambda_{i}}|y-y_{i}|^{2\sigma-n+\delta})^{2}\\&
\leq m_i^{-2\lda_i}\int_{B_\rho(y_i)\setminus B_{r_i}(y_i)}|y-y_i|(|y-y_{i}|^{2\sigma-n+\delta})^{2}\\&
 =\begin{cases}
O(m_i^{-2\lambda_{i}} ),\quad& n<2(2\sigma+\delta)+1,\\
O(m_i^{-2\lambda_{i}} )\ln m_i,\quad&n=2(2\sigma+\delta)+1,\\
O(m_i^{\frac{-4\sigma-2}{n-2\sigma}+o(1)}),\quad & n>2(2\sigma+\delta)+1,
\end{cases}
\end{split}
\ee
and
\[
\rho\int_{\pa B_\rho} \frac{1}{2} aU^2+\frac{1}{p_i+1}U^{p_i+1}\,\ud S=O(m_i^{-2+\frac{4\delta}{n-2\sigma}+o(1)}).
\]
By Lemma \ref{lem:up-1}, Lemma \ref{lem:s-harnack} and regularity theory of linear equations in \cite{JLX},
\[
P_\sigma(Y_i,\rho ,U_i)=O(m_i^{-2+\frac{4\delta}{n-2\sigma}+o(1)}).
\]

By Proposition \ref{prop:blowupbubble}, we have
\be\label{4.1333}
\begin{split}
&\int_{B_{r_i}(y_i)}|y-y_i|u_i^{2}\\&
\leq C\int_{B_{r_i}(y_i)}\frac{|y-y_i|m_i^{2}}{(1+\bar{c}|m_i^{(p_i-1)/2\sigma}(y-y_i)|^2)^{(n-2\sigma)}}\\&
\leq C m_i^{2-(n+1)\frac{p_{i}-1}{2\sigma}}\int_{B_{R_i}}\frac{|z|}{(1+\bar{c}|z|^2)^{n-2\sigma}}\\&
=\begin{cases}
O(m_i^{2-(n+1)\frac{p_{i}-1}{2\sigma}} )=O(m_i^{\frac{-4\sigma-2}{n-2\sigma}+o(1)} ),\quad& n>4\sigma+1,\\
O(m_i^{2-(n+1)\frac{p_{i}-1}{2\sigma}} )\ln m_i=O(m_i^{\frac{-4\sigma-2}{n-2\sigma}+o(1)} )\ln m_i,\quad& n=4\sigma+1,\\
O(m_i^{2-(n+1)\frac{p_{i}-1}{2\sigma}})\times R^{n+1-2(n-2\sigma)}_{i}=o(m_i^{-2+o(1)}),\quad & n<4\sigma+1.
\end{cases}
\end{split}
\ee
Since $a_i\ge 0$, combining the above estimates and the fact $\tau_i=o(1)$, the lemma follows immediately.

\end{proof}

\begin{lem}\label{lem:toup}Assume as in Lemma \ref{lem:up-1}.  Then for all $0<\theta<1$, we have
\[
\limsup_{i\rightarrow\infty}\max_{y\in\partial B_{\theta}(y_{i})}u_{i}(y)u_{i}(y_{i})\leq C(\theta).
\]

\end{lem}

\begin{proof}
By Lemma  \ref{lem:s-harnack}, it suffices to show the lemma for sufficiently small $\theta>0$.
Let $e=(e_1,\dots, e_{n+1})$ be a unit vector with $e_{n+1}=0$, $Y_{\theta}=Y_i+\theta e$ and $\xi_{i}(Y)=U_{i}(Y_{\theta})^{-1}U_{i}(Y)$. Then $\xi_{i}(Y)$ satisfies
\[
\begin{cases}
\mathrm{div}(t^{1-2\sigma}\nabla \xi_{i}(Y))=0&\quad \mbox{in }\mathcal{B}_{3}^+,\\
\frac{\pa}{\pa \nu^\sigma} \xi_{i}(y,0)=a_i(y)\xi_{i}(y,0)+ U_{i}(Y_{\theta})^{p_{i}-1}\xi_{i}(y,0)^{p_{i}}&\quad \mbox{on }\pa' \mathcal{B}_{3}^+.
\end{cases}
\]
It follows from lemma \ref{lem:s-harnack} that for any compact set $K\subseteq \mathcal{B}_{1}^+\setminus \{0\}$ ,
\[
C(K)^{-1}\leq\xi_{i}(Y)\leq C(K) \quad \mbox{on } K,
\]
where $C(K)>0$ depends only on $n,\sigma, A_0,A_1$ and $K$.
Note also that $U_{i}(Y_{i}+\theta e)\rightarrow0 $ as $i\rightarrow\infty$ by Lemma \ref{lem:up-1}. Then  after passing to a subsequence,
\begin{align} \label{eq:lem49-a}
\xi_i-\xi,  \nabla_x(\xi_{i}- \xi) \mbox{ and }s^{1-2\sigma} \pa_s( \xi_i-\xi) \mbox{ converge to } 0 \quad \mbox{in }C^{\al}(\B_1^+\setminus \B_\va^+)
\end{align}
for some $\al\in (0,1)$ and every $\va>0$,
\begin{align*}
\xi_i(y,0)\to  \xi(y,0) \quad \mbox{in } C^2_{loc}(B_1\setminus \{0\}),
\end{align*}
for some  $\xi$ satisfying
\[
\begin{cases}
\mathrm{div}(s^{1-2\sigma}\nabla \xi(Y))=0&\quad \mbox{in }\mathcal{B}_{1/2}^+,\\
\frac{\pa }{\pa \nu^\sigma} \xi(y,0)=a(y)\xi(y,0)
&\quad \mbox{on }\pa' \mathcal{B}_{1/2}^+\setminus \{0\},
\end{cases}
\]
with $a(y)=\lim\limits_{i\rightarrow\infty} a_i(y)$. Hence $\lim\limits_{i\rightarrow\infty}u_{i}(y_{\theta})^{-1}r^{\frac{2\sigma}{p_{i}-1}}\bar u_{i}(r)=r^{\frac{n-2\sigma}{2}}\bar\xi(r,0)$, where $\bar \xi(r,0)$ is the integral average of $\xi(\cdot,0)$ over $\pa B_r$.
Since $r_i\to 0$ and $y_i\to0$ is an isolated simple blow up point of $\{u_i\}$, it follows from Proposition \ref{prop:blowupbubble}  that $r^{(n-2\sigma)/2}\bar \xi(r,0)$
is nonincreasing  for all $0<r<\rho$, i.e.,  for any $0<r_1\leq r_2<\rho$,
\[
r_1^{(n-2\sigma)/2}\bar \xi (r_1,0)\geq r_2^{(n-2\sigma)/2}\bar \xi(r_2,0).
\]
 Therefore, $\xi$ has to have a singularity at $Y=0$.
By Proposition \ref{prop:bocher},
\be\label{4.16}
\xi(Y)=A |Y|^{2\sigma -n}+O(|Y|^{4\sigma-n}), \quad 0<|Y|\leq 1/2,
\ee
where $A> 0$. For any given $0<d<1/2$, let $\phi>0$ be the first eigenfunction of
\[
\begin{cases}
\mathrm{div}(s^{1-2\sigma}\nabla \phi(Y))=0&\quad \mbox{in }\mathcal{B}_{d}^+,\\
-\frac{\pa }{\pa \nu^\sigma} \phi (y,0)=\lda_1\phi(y,0) &\quad \mbox{on }\pa' \mathcal{B}_{d}^+,\\
\phi=0 & \quad \mbox{on }\pa'' \mathcal{B}_{d}^+.
\end{cases}
\]
Let $d$ be small so that $\lda_1\ge A_0$. Let $W_i=\frac{\xi_i}{\phi}$. Then
\begin{align} \label{eq:conformal}
\begin{cases}
\mathrm{div}(s^{1-2\sigma}\phi^2\nabla W_i(Y))=0&\quad \mbox{in }\mathcal{B}_{d}^+,\\
-\lim_{s\to 0}s^{1-2\sigma} \phi^2 \pa_s W_i (y,0)=(a_i-\lda_1)\phi \xi_i+\phi U_i(Y_\theta)^{p_i-1}\xi_i^{p_i} &\quad \mbox{on }\pa' \mathcal{B}_{d}^+,
\end{cases}
\end{align}
in weak sense. It follows that for $\theta\in(0,\frac{d}{2}]$,
\be\label{4.19}
\begin{split}
&0=\int_{\pa''\B_\theta^+(Y_i)}s^{1-2\sigma}\phi^2 \frac{\pa W_i}{\pa \nu}+\int_{\pa'\B_\theta^+(Y_i)}(a_i-\lda_1)\phi \xi_i+\phi U_i(Y_\theta)^{p_i-1}\xi_i^{p_i} .
\end{split}
\ee
By \eqref{eq:lem49-a} and \eqref{4.16}, we have for $i$ large
\[\begin{split}
-\int_{\pa''\B_\theta^+(Y_i)}s^{1-2\sigma}\phi^2\frac{\pa W_i}{\pa \nu}&\ge
-A \int_{\pa''\B_\theta^+(0)}s^{1-2\sigma} \phi \frac{\pa}{\pa \nu}|Y|^{2\sigma -n}- C\theta^{2\sigma} \\&
=A(n-2\sigma)\theta^{2\sigma-n-1}\min_{B_{d/2}}\phi \int_{\pa''\B_\theta^+}s^{1-2\sigma}- C\theta^{2\sigma}=:m(\theta)>0,
\end{split}
\]
provided $\theta$ is small.
By Proposition \ref{prop:blowupbubble}, Lemma \ref{lem:up-1} and Lemma \ref{lem:tau-1} we have
\[
\begin{split}
\int_{\pa'\B_{\theta}^+(Y_i)}U_i^{p_i}\leq CU_i(Y_i)^{-1}.
\end{split}
\]
It follows that
\begin{align*}
m(\theta) &\le \int_{\pa'\B_\theta^+(Y_i)}(a_i-\lda_1)\phi \xi_i+\phi U_i(Y_\theta)^{p_i-1}\xi_i^{p_i} \\&
\le \int_{\pa'\B_\theta^+(Y_i)}\phi U_i(Y_\theta)^{p_i-1}\xi_i^{p_i}\\&
=U_i(Y_\theta)^{-1}  \int_{\pa'\B_\theta^+(Y_i)}\phi U_i^{p_i}\le CU_i(Y_\theta)^{-1}  U_i(Y_i)^{-1}.
\end{align*}
Thus
\[
U_i(Y_\theta) U_i(Y_i)\le \frac{C}{m(\theta)}.
\]
Therefore, we complete the proof.
\end{proof}

\begin{prop}\label{prop:up-1}Assume as in Lemma \ref{lem:up-1}. Then
\[
u_{i}(y)\leq Cu_{i}(y_{i})^{-1}|y-y_{i}|^{2\sigma-n} \quad \forall ~ |y-y_{i}|\leq 1,
\]
where $C\ge 0$ depends only on  $n,\sigma, A_0, A_1$ and $\rho$.
\end{prop}

\begin{proof} It suffices to show
\be\label{7.77}
\begin{split}
 U_{i}(Y)U_{i}(Y_{i})|Y-Y_{i}|^{n-2\sigma}\leq C.
\end{split}
\ee
If not, then after passing to a subsequence we can find $\{\tilde{Y_{i}}\}$ such that $|\tilde{Y_{i}}-Y_{i}|\leq 1$ and
\be\label{7.78}
\begin{split}
 U_{i}(\tilde{Y_{i}})U_{i}(Y_{i})|\tilde{Y_{i}}-Y_{i}|^{n-2\sigma}\rightarrow \infty, ~as~i\rightarrow\infty.
\end{split}
\ee
It follows from \eqref{limit2}  that
\[
r_{i}=R_{i}u_{i}(y_{i})^{-\frac{p_{i}-1}{2\sigma}}\leq| \tilde{Y_{i}}-Y_{i}|\leq 1.
\]
Set $\mu_i=| \tilde{Y_{i}}-Y_{i}|$, $\tilde{U_{i}}(Y)=\mu_i^{\frac{2\sigma}{p_{i}-1}}U_{i}(\mu_iY+Y_{i})$. Clearly, $\tilde{U_{i}}(Y)$ satisfies
\[
\begin{cases}
\mathrm{div}(s^{1-2\sigma}\nabla \tilde{U_{i}}(Y))=0,&\quad \mbox{in }\mathcal{B}_{1}^+\\
 \frac{\pa }{\pa \nu^\sigma}\tilde{U_{i}}(y,0)=\tilde{a_i}\tilde{U_{i}} +\tilde{U_{i}}^{p_{i}},&\quad \mbox{on }\pa' \mathcal{B}_{1}^+.
\end{cases}
\]
where $ \tilde{a_i}(Y)=\mu_i^{2\sigma}a_i(\mu_iY+Y_{i})$.
It is easy to see that $\tilde{U_{i}}(Y)$ satisfies the hypotheses of Lemma \ref{lem:toup} and therefore
\[
\max_{|Y|=1}\tilde{U_{i}}(Y)\tilde{U_{i}}(0)\leq C,
\]
from which we deduce that
\[
 U_{i}(Y_{i}) U_{i}(\tilde{Y_{i}})\mu_i^{n-2\sigma}\leq C.
 \]
Namely,
\[
U_{i}(Y_{i}) U_{i}(\tilde{Y_{i}})|\tilde{Y_{i}}-Y_{i}|^{n-2\sigma}\leq C,
\]
which contradicts \eqref{7.78}. We thus established \eqref{7.77} and the proof of the proposition is finished.

\end{proof}

\begin{cor}\label{cor:conclusion}Assume as in Lemma \ref{lem:up-1}. We have
\[
\begin{split}
\int_{|y-y_i|\leq 1}&|y-y_i|^{s}u_i(y)^{2}\,\ud y=\begin{cases}
O(m_i^{-2}),\quad& s+4\sigma >n,\\
O(m_i^{-2})\ln m_i,\quad& s+4\sigma =n,\\
O(m_i^{\frac{-4\sigma-2s}{n-2\sigma}}),\quad & s+4\sigma <n.
\end{cases}
\end{split}
\]
\end{cor}

\begin{proof} It follows from Proposition \ref{prop:blowupbubble}, Lemma \ref{lem:tau-1} and Proposition \ref{prop:up-1}.
\end{proof}

\begin{lem}\label{lem:a_i} Assume as in Lemma \ref{lem:up-1}. Then
\[
a_i(y_i)\le C\begin{cases}
(\ln m_i)^{-1}(1+\|\nabla^2 a_i\|_{L^\infty(B_1)})& \quad \mbox{if }n=4 \sigma,\\
m_i^{-2+\frac{4\sigma}{n-2\sigma}}(1+\|\nabla^2 a_i\|_{L^\infty(B_1)})& \quad \mbox{if }4\sigma < n<4 \sigma+2,\\
m_i^{-2+\frac{4\sigma}{n-2\sigma}} (1+\|\nabla^2 a_i\|_{L^\infty(B_1)}\ln m_i) & \quad \mbox{if }n=4 \sigma+2,\\
m_i^{-2+\frac{4\sigma}{n-2\sigma}}  + \|\nabla^2 a_i\|_{L^\infty(B_1)} m_i^{-\frac{4}{n-2\sigma}} & \quad \mbox{if }n>4 \sigma+2,
\end{cases}
\]
and
\[
|\nabla a_i(y_i)|\le C\begin{cases}
(\ln m_i)^{-1}(1+\|\nabla^2 a_i\|_{L^\infty(B_1)})& \quad \mbox{if }n=4 \sigma,\\
m_i^{-2+\frac{4\sigma}{n-2\sigma}}(1+\|\nabla^2 a_i\|_{L^\infty(B_1)})& \quad \mbox{if }4\sigma <n<4 \sigma+1,\\
m_i^{-2+\frac{4\sigma}{n-2\sigma}} (1+\|\nabla^2 a_i\|_{L^\infty(B_1)}\ln m_i) & \quad \mbox{if }n=4 \sigma+1,\\
m_i^{-2+\frac{4\sigma}{n-2\sigma}}  + \|\nabla^2 a_i\|_{L^\infty(B_1)} m_i^{-\frac{2}{n-2\sigma}} & \quad \mbox{if }n>4 \sigma+1,
\end{cases}
\]
where $C>0$ depends only on $n,\sigma, A_0, A_1$ and $\rho$.
\end{lem}

\begin{proof} Choose a cut-off function $\eta\in C^\infty_c(\B_{1/2})$ satisfying
$\eta(Y)=1 $ if $|Y|\le 1/4$. Multiplying \eqref{eq:mainseq-ext} by $\eta(Y-Y_i)\pa_{y_j} U_i(y,s)$, $j=1,\dots, n$, and integrating by parts over $\B_1^+$, we obtain
\begin{align*}
0&=-\int_{\B_1^+} s^{1-2\sigma}\nabla U_i\nabla (\eta \pa y_j U_i)+ \int_{\pa'\B_1^+}\eta\pa_{y_j} U_i(a_i U_i+ U_i^{p_i})\\&
=\frac{1}{2}\int_{\B_{1/2}^+\setminus \B_{1/4}^+} s^{1-2\sigma}[|\nabla U_i|^2\pa_{y_j}\eta- 2\nabla U_i \nabla \eta \pa_{y_j} U_i]- \int_{\pa' \B_1^+}[\frac{1}{2} \pa_{y_j}(a_i \eta) U_i^2+\frac{1}{p_i+1}\pa_{y_j} \eta U_i^{p_i+1}].
\end{align*}
By Proposition \ref{prop:up-1}, we have
\[
U_i(Y)\le C U_i(Y_i)^{-1} \quad \forall~\frac14\le |Y|\le \frac12
\]
and
\[
\int_{\B_{1/2}^+\setminus \B_{1/4}^+} s^{1-2\sigma}|\nabla U_i|^2\le C U_i(Y_i)^{-2}.
\]
Hence, by Corollary \ref{cor:conclusion},
\begin{align}
\left| \pa_{j} a_i(y_i) \int_{B_1} u_i^2 \right|& \le C U_i(Y_i)^{-2}+ \int_{B_1} |\pa_j a_i(y_i)-\pa_j a_i(y)|u_i^2 \nonumber
\\& \le C\begin{cases}
m_i^{-2}(1+\|\nabla^2 a_i\|_{L^\infty(B_1)})& \quad \mbox{if }n<4 \sigma+1,\\
m_i^{-2}(1+\|\nabla^2 a_i\|_{L^\infty(B_1)}\ln m_i) & \quad \mbox{if }n=4 \sigma+1,\\
m_i^{-2} + \|\nabla^2 a_i\|_{L^\infty(B_1)} m_i^{-\frac{4\sigma+2}{n-2\sigma}} & \quad \mbox{if }n>4 \sigma+1 .
\end{cases}
\label{eq:gradientest}
\end{align}
By Proposition \ref{prop:low-1} and Lemma \ref{lem:tau-1},
\be \label{eq:l2}
\int_{B_1} u_i^2 \ge \frac{1}{C}  m_i^{-\frac{4\sigma}{n-2\sigma}} \int_{B_{m_i^{\frac{p_i-1}{2\sigma}}}(0)}\frac{1}{(1+|x|^2)^{n-2\sigma}}\,\ud x.
\ee
Therefore, desired estimates of $|\nabla a_i(y_i)|$ follows.

By \eqref{4.11}, using Lemma \ref{lem:tau-1} and Proposition \ref{prop:up-1} we have
\begin{align*}
\tau_i &\le C\int_{B_{1}(y_i)}|y-y_i||\nabla a_i(y)| U_i(y,0)^2 \,\ud y+C m_i^{-2}\\&
\le C |\nabla a_i(y_i)| \int_{B_{1}(y_i)}|y-y_i| U_i(y,0)^2 \,\ud y \\ &\quad +C \|\nabla^2 a_i\|_{L^\infty(B_1)}  \int_{B_{1}(y_i)}|y-y_i|^2 U_i(y,0)^2 \,\ud y+C m_i^{-2} \\&
\le C \begin{cases}
m_i^{-2}(1+\|\nabla^2 a_i\|_{L^\infty(B_1)})& \quad \mbox{if }n<4 \sigma+2,\\
m_i^{-2}(1+\|\nabla^2 a_i\|_{L^\infty(B_1)}\ln m_i) & \quad \mbox{if }n=4 \sigma+2,\\
m_i^{-2} + \|\nabla^2 a_i\|_{L^\infty(B_1)} m_i^{-\frac{4\sigma+4}{n-2\sigma}} & \quad \mbox{if }n>4 \sigma+2 .
\end{cases}
\end{align*}
where we used $a_i\ge 0$ in the first inequality.

Using \eqref{4.11} again, by the estimates for $|\nabla a_i(y_i)|$, $\tau_i$ and estimates \eqref{eq:l2},  the estimate of $a_i(y_i)$ follows immediately.
\end{proof}

\section{Expansions of blow up solutions}

\label{sec:5}

\begin{lem}\label{lem:aux} For $s\ge 0$, $\ell>100$, $0<\al<n$ and $\al\le \mu$, we have
\[
\int_{|y|\le \ell} \frac{1}{(|x-y|^2+s^2)^{\frac{n-\al}{2}}}\frac{1}{(1+|y|)^\mu}\,\ud y\le C \begin{cases} \ln(\frac{\ell}{r}+1)&  \quad \mbox{if }\mu=\al,\\
(1+r)^{\al-\mu}& \quad \mbox{if }\al<\mu<n,\\
(1+r)^{\al-n}\ln (2+r) & \quad \mbox{if }\mu=n,\\
(1+r)^{\al-n} & \quad \mbox{if }\mu>n,
\end{cases}
\]
for all $r=\sqrt{|x|^2+s^2}<\ell$, where $C>0$ is independent of $\ell$.

\end{lem}

\begin{proof} Let $r^2= |x|^2+s^2$. Then by change of variables $y=rz$ we have
\begin{align*}
&\int_{|y|\le \ell} \frac{1}{(|x-y|^2+s^2)^{\frac{n-\al}{2}}}\frac{1}{(1+|y|)^\mu}\,\ud y\\& = r^{\al}
\int_{|z|\le \ell/r} \frac{1}{(|x/r-z|^2+s^2/r^2)^{\frac{n-\al}{2}}}\frac{1}{(1+r|z|)^\mu}\,\ud z
\\& =r^{\al}
\int_{|z|\le 1/10}+ \int_{\frac{1}{10}\le |z|\le \ell/r}\frac{1}{(|x/r-z|^2+s^2/r^2)^{\frac{n-\al}{2}}}\frac{1}{(1+r|z|)^\mu}\,\ud z
\\&\le C r^{\al} \int_{|z|\le 1/10} \frac{1}{(1+r|z|)^\mu}\,\ud z +Cr^{\al-\mu}\int_{\frac{1}{10}\le |z|\le \ell/r}\frac{1}{(|x/r-z|^2+s^2/r^2)^{\frac{n-\al}{2}}}\frac{1}{|z|^\mu}\,\ud z.
\end{align*}
The lemma follows immediately.

\end{proof}

Let
\[
\theta_\lda(x)=  \left(\frac{\lda}{1+\lda^2\bar c|x|^2}\right)^{\frac{n-2\sigma}{2}}
\]
and
\[
\Theta_\lda(x,t)= \mathcal{P}_\sigma*\theta_\lda(x,t),
\]
where $\bar c$ is chosen such that $(-\Delta)^\sigma \theta_\lda= \theta_\lda^\frac{n+2\sigma}{n-2\sigma}$ as in Proposition \ref{prop:blowupbubble}. In the following we will adapt some arguments from Marques \cite{Marques} for the Yamabe equation;
see also the proof of Proposition 2.2 of Li-Zhang \cite{Li-Zhang05}.

\begin{lem}\label{lem:s-1} Assume as in Lemma \ref{lem:up-1}. Suppose $\rho=1$.  If $n\ge 4\sigma$, we have
for $|Y|\le  m_i^{\frac{p_i-1}{2\sigma}} $\[
|\Phi_i(Y)-\Theta_1 (Y)|\le C \begin{cases}
m_i^{-2}(1+\|\nabla^2 a_i\|_{L^\infty(B_1)})& \quad \mbox{if }n<4 \sigma+2,\\
m_i^{-2}(1+\|\nabla^2 a_i\|_{L^\infty(B_1)}\ln m_i) & \quad \mbox{if }n=4 \sigma+2,\\
m_i^{-2} + \|\nabla^2 a_i\|_{L^\infty(B_1)} m_i^{-2+\frac{2(n-(4\sigma+2))}{n-2\sigma}} & \quad \mbox{if }n>4 \sigma+2 ,
\end{cases}
\]
where $\Phi_i(Y)=\frac{1}{m_i}U_i(m_i^{-\frac{p_i-1}{2\sigma}}Y+Y_i)$, $m_i=u_i(0)$, and $C>0$ depends only on $n, \sigma$ and $ A_0$.

\end{lem}

\begin{proof} For brevity, set $\ell_i= m_i^{\frac{p_i-1}{2\sigma}} $. Let
\[
\Lda_i=\max_{|Y|\le \ell_i }|\Phi_i(Y)-\Theta_1(Y)|.
\]
By Proposition \ref{prop:up-1} and Lemma \ref{lem:s-harnack}, we have for any $0<\va<1$ and $\va \ell_i\le |Y|\le \ell_i$
\[
|\Phi_i(Y)-\Theta_1(Y)|\le C(\va) m_i^{-2},
\]
where we used $m_i^{\tau_i}=1+o(1)$. Hence, we may assume that $\Lda_i$ is achieved at some point $|Z_i|\le \frac12 \ell_i$, otherwise the proof is finished. By maximum principle, $Z_i=(z_i,0)$.  Set
\[
V_i(Y)=\frac{1}{\Lda_i} (\Phi_i(Y)-\Theta_1(Y)).
\]
Then
\[
\begin{cases}
\mathrm{div}(s^{1-2\sigma}\nabla V_i(Y))=0&\quad \mbox{in }\mathcal{B}_{\ell_i}^+,\\
\frac{\pa }{\pa \nu^\sigma} V_i(y,0)=b_iV_i(y,0)+\frac{\tilde a_i}{\Lda_i}\Phi_i(y,0) & \mbox{on }\pa' \B_{\ell_i}^+,
\end{cases}
\]
where $\tilde a_i(y)= m_i^{1-p_i} a_i(\ell_i^{-1} y+y_i)$ and
\[
b_i(y)=\frac{\Phi_i(y,0)^{p_i}-\theta_1(y)^{p_i}}{\Phi_i(y,0)-\theta_1(y)}.
\]
Let
\be\label{eq:contrad-1}
W_i(Y):= c(n,\sigma)\int_{|z|\le \ell_i} \frac{b_iV_i(z,0)+\frac{\tilde a_i}{\Lda_i}\Phi_i(z,0)}{(|y-z|^2+s^2)^{\frac{n-2\sigma}{2}}}\,\ud z,
\ee
where $c(n,\sigma)$ is the constant in \eqref{eq:flat-neumann}.
Then $W_i(Y)\in W^{1,2}(s^{1-2\sigma}, \B_{\ell_i}^+)$ is a weak solution of
\[
\begin{cases}
\mathrm{div}(s^{1-2\sigma}\nabla W_i(Y))=0&\quad \mbox{in }\mathcal{B}_{\ell_i}^+,\\
\frac{\pa }{\pa \nu^\sigma} W_i(y,0)=b_iV_i(y,0)+\frac{\tilde a_i}{\Lda_i}\Phi_i(y,0) & \mbox{on }\pa' \B_{\ell_i}^+.
\end{cases}
\]
By Taylor expansion of $a_i$ at $y_i$, we have
\begin{align*}
a_i(\ell_i^{-1}y +y_i)&\le a_i(y_i)+\ell_i^{-1}| y| |\nabla a_i(y_i)|+ \ell_i^{-2}|y|^2\|\nabla^2 a_i\|_{L^\infty(B_1)}.
\end{align*}
Since $\Phi_i(y,0)\le C \theta_1(y)$, by Lemma \ref{lem:aux} and Lemma \ref{lem:a_i} we have
\[
\int_{|z|\le \ell_i} \frac{\tilde a_i\Phi_i(z,0)}{(|y-z|^2+s^2)^{\frac{n-2\sigma}{2}}}\,\ud z\le C \al_i
\]
with
\be \label{eq:alpha_i}
\al_i:= \begin{cases}
m_i^{-2}(1+\|\nabla^2 a_i\|_{L^\infty(B_1)})& \quad \mbox{if }n<4 \sigma+2,\\
m_i^{-2}(1+\|\nabla^2 a_i\|_{L^\infty(B_1)}\ln m_i) & \quad \mbox{if }n=4 \sigma+2,\\
m_i^{-2} + \|\nabla^2 a_i\|_{L^\infty(B_1)} m_i^{-2+\frac{2(n-(4\sigma+2))}{n-2\sigma}} & \quad \mbox{if }n>4 \sigma+2,
\end{cases}
\ee
and
\[
b_i(y)\le C (1+|y|^2)^{-3\sigma/2} .
\]
It follows that
\be \label{eq:ess-v_i}
|W_i(Y) |\le C (1+|Y|)^{-\sigma}+C \frac{\al_i}{\Lda_i}.
\ee
Note that
\be \label{eq:w-v}
\begin{cases}
\mathrm{div}(s^{1-2\sigma}\nabla  (V_i-W_i))=0&\quad \mbox{in }\mathcal{B}_{\ell_i}^+,\\
\frac{\pa }{\pa \nu^\sigma}(V_i-W_i) =0 & \mbox{on }\pa' \B_{\ell_i}^+.
\end{cases}
\ee
If Lemma \ref{lem:s-1} were wrong, by maximum principle we have
\be \label{eq:err-v_i}
\|W_i-V_i\|_{L^\infty(\B_{\ell_i/2}^+)}\le \sup_{\pa'' \B_{\ell_i/2}^+} |W_i-V_i|\le C (\ell_i^{-\sigma}+\frac{\al_i+m_i^{-2}}{\Lda_i})\to 0
\ee
as $i\to \infty$. By regularity theory in \cite{JLX}, both $W_i(y,0)$ and $V_i(y,0)$ are locally uniformally bounded in $C^{2+\va}$ for some $0<\va<1$. By \eqref{eq:contrad-1}, it follows from Arzela-Ascoli theorem and Lebesgue dominated convergence theorem that, after passing to subsequence,
\[
W_i(y,0), ~ V_i(y,0) \to v(y) \quad \mbox{in } C^{2}_{loc}(\R^n)
\]
for some $v\in C^2_{loc}(\R^n)\cap L^\infty(\R^n)$ satisfying
\[
v(y)= C(n,\sigma)\frac{n+2\sigma}{n-2\sigma} \int_{\R^n}\frac{\theta_1(z)^{\frac{4\sigma}{n-2\sigma}}v(z)}{|y-z|^{n-2\sigma}}\,\ud z.
\]
It follows from the non-degeneracy result, see, e.g., the proof of Lemma 4.1 of \cite{Li-Xiong}, that
\[
v(y)=c_0 (\frac{n-2\sigma}{2} \theta_1 +y\nabla \theta_1) +\sum_{j=1}^n c_j \pa_j \theta_1,
\]
where $c_0,\dots,  c_n$ are constants. By Proposition \ref{prop:blowupbubble}, $v(0)=0$ and $\nabla v(0)=0$. Hence, $v=0$. Hence, the maximum points $(z_i,0)$ of $V_i$ have to go to infinity. This contradicts to \eqref{eq:ess-v_i} and \eqref{eq:err-v_i}.
\end{proof}

\begin{lem}\label{lem:s-2} After passing to subsequence, if necessary, there holds
 \begin{align*}
&|\Phi_i(Y)-\Theta_1 (Y)|\\&\le C   \begin{cases}  \max\{\|\nabla^2 a_i\|_{L^\infty(B_1)}m_i^{-2+\frac{2\sigma}{n-2\sigma}}(1+|Y|)^{-\sigma}, m_i^{-2}\} &\quad \mbox{if }n=4\sigma+2,\\
 \max\{ \|\nabla^2 a_i\|_{L^\infty(B_1)} m_i^{-2+\frac{2(n-(4\sigma+2))}{n-2\sigma}}(1+|Y|)^{4\sigma+2-n}, m_i^{-2}\} &\quad \mbox{if }n>4\sigma+2.
\end{cases}
\end{align*}

\end{lem}

\begin{proof} Let $\al_i$ be defined in \eqref{eq:alpha_i}. We may assume $\frac{m_i^{-2}}{ \al_i}\to 0$ as $i\to \infty$ for $n\ge 4\sigma+2$; otherwise there exists a subsequence $i_l$ of $\{i\}$ such that $m^{-2}_{i_l}\ge \frac{1}{C} \al_{i_l}$ for some $C>0$ and the lemma follows  from Lemma \ref{lem:s-1}.   Set
\[
\al_i'=\begin{cases}
\|\nabla^2 a_i\|_{L^\infty(B_1)}m_i^{-2+\frac{2\sigma}{n-2\sigma}} \quad \mbox{if }n=4\sigma+2,\\
\|\nabla^2 a_i\|_{L^\infty(B_1)} m_i^{-2+\frac{2(n-(4\sigma+2))}{n-2\sigma}} \quad \mbox{if }n>4\sigma+2,
\end{cases}
\]
and
\[
V_i(Y)= \frac{\Phi_i(Y)-\Theta_1(Y)}{\al_i'}, \quad |Y|\le m_i^{\frac{p_i-1}{2\sigma}}.
\]
Since $\frac{m_i^{-2}}{ \al_i}\to 0$ and $\al_i\le \al_i'$, it follows from Lemma \ref{lem:s-1} that $|V_i|\le C$.  Since $0<\Phi_{i}\le C\Theta_1$ and $m_i^{\tau_i}=1+o(1)$, we have
 \be \label{eq:outside}
 |V_i(Y)|\le \begin{cases} Cm_i^{-\frac{2\sigma}{n-2\sigma}}& \quad \mbox{if }n=4\sigma+2,\\
C m_i^{-\frac{2(n-(4\sigma+2))}{n-2\sigma}}& \quad \mbox{if }n>4\sigma+2,
 \end{cases} \quad \frac12 \ell_i \le |Y|\le \ell_i
 \ee and thus we  only need to prove the proposition when $|Y|\le \frac{1}{2} \ell_i$, where $\ell_i=m_i^{\frac{p_i-1}{2\sigma}}$.
Note that
\[
\begin{cases}
\mathrm{div}(s^{1-2\sigma}\nabla V_i(Y))=0&\quad \mbox{in }\mathcal{B}_{\ell_i}^+,\\
\frac{\pa }{\pa \nu^\sigma} V_i(y,0)=b_iV_i(y,0)+\frac{\tilde a_i}{\al_i'}\Phi_i(y,0) & \mbox{on }\pa' \B_{\ell_i}^+
\end{cases}
\]
in weak sense, where $\tilde a_i(y)= m_i^{1-p_i} a_i(\ell_i^{-1} y+y_i)$ and
\[
b_i(y)=\frac{\Phi_i(y,0)^{p_i}-\theta_1(y)^{p_i}}{\Phi_i(y,0)-\theta_1(y)}.
\]
Let
\be\label{eq:contrad-2}
W_i(Y):= C(n,\sigma)\int_{|z|\le \ell_i} \frac{b_iV_i(z,0)+\frac{\tilde a_i}{\al_i'}\Phi_i(z,0)}{(|y-z|^2+s^2)^{\frac{n-2\sigma}{2}}}\,\ud z.
\ee
Then $W_i(Y)\in W^{1,2}(s^{1-2\sigma}, \B_{\ell_i}^+)$ is a weak solution of
\[
\begin{cases}
\mathrm{div}(s^{1-2\sigma}\nabla W_i(Y))=0&\quad \mbox{in }\mathcal{B}_{\ell_i}^+,\\
\frac{\pa }{\pa \nu^\sigma} W_i(y,0)=b_iV_i(y,0)+\frac{\tilde a_i}{\al_i'}\Phi_i(y,0) & \mbox{on }\pa' \B_{\ell_i}^+.
\end{cases}
\]
By Taylor expansion of $a_i$ at $y_i$, we have
\begin{align*}
a_i(\ell_i^{-1}y +y_i)&\le a_i(y_i)+\ell_i^{-1}| y| |\nabla a_i(y_i)|+ \ell_i^{-2}|y|^2\|\nabla^2 a_i\|_{L^\infty(B_1)}.
\end{align*}
Since $\Phi_i(y,0)\le C \theta_1(y)$, by Lemma \ref{lem:aux} and Lemma \ref{lem:a_i} we have:
for $n=4\sigma+2$
\begin{align*}
\int_{|z|\le \ell_i} \frac{\frac{\tilde a_i}{\al_i'}\Phi_i(z,0)}{(|y-z|^2+s^2)^{\frac{n-2\sigma}{2}}}\,\ud z&\le C \int_{|z|\le \ell_i} \frac{ 1}{(|y-z|^2+s^2)^{\frac{n-2\sigma}{2}}(1+|z|)^{n-2\sigma-2}m_i^{\frac{2\sigma}{n-2\sigma}}}\,\ud z\\&
\le C  \int_{|z|\le \ell_i} \frac{ 1}{(|y-z|^2+s^2)^{\frac{n-2\sigma}{2}}(1+|z|)^{n-\sigma-2}}\,\ud z\\&
\le C(1+|Y|)^{-\sigma};
\end{align*}
for $n>4\sigma+2$
\begin{align*}
\int_{|z|\le \ell_i} \frac{\frac{\tilde a_i}{\al_i'}\Phi_i(z,0)}{(|y-z|^2+s^2)^{\frac{n-2\sigma}{2}}}\,\ud z&\le C \int_{|z|\le \ell_i} \frac{ 1}{(|y-z|^2+s^2)^{\frac{n-2\sigma}{2}}(1+|z|)^{n-2\sigma-2}}\,\ud z\\&
\le C(1+|Y|)^{4\sigma+2-n}.
\end{align*}
Since
\[
b_i(y)\le C (1+|y|^2)^{-3\sigma/2} ,
\] and $V_i(y,0)\le C $, by Lemma \ref{lem:aux} we have
\begin{align}
\label{eq:in-1}
\int_{|z|\le \ell_i} \frac{|b_iV_i(z,0)|}{(|y-z|^2+s^2)^{\frac{n-2\sigma}{2}}}\,\ud z\le C (1+|Y|)^{-\sigma}.
\end{align}
Hence, we obtain
\[
|W_i(Y)| \le C(1+|Y|)^{-\sigma}.
\]
Since $W_i-V_i$ satisfies the homogenous equation \eqref{eq:w-v}, by \eqref{eq:outside} and the maximum principle  we have
\begin{align} \label{eq:in-3}
|V_i(Y)|&\le |W_i(Y)|+\max_{|Y|=\ell_i/2}|W_i-V_i|\\& \le C (1+|Y|)^{-\sigma}+ C m_i^{-\frac{2\sigma}{n-2\sigma}}\le C (1+|Y|)^{-\sigma} \quad \mbox{for }|Y|\le \ell_i/2.
\nonumber
\end{align}
Therefore, we proved the lemma when $n=4\sigma+2$. If $n> 4\sigma+2$, we use above estimate of $V_i$ and can improve \eqref{eq:in-1} to
\begin{align}
\label{eq:in-2}
\int_{|z|\le \ell_i} \frac{|b_iV_i(z,0)|}{(|y-z|^2+s^2)^{\frac{n-2\sigma}{2}}}\,\ud z\le C (1+|Y|)^{-2\sigma}.
\end{align}
It follows that
\[
|W_i(Y)|\le  C (1+|Y|)^{-2\sigma}+ C(1+|Y|)^{4\sigma+2-n}.
\]
Arguing as \eqref{eq:in-3}, $V_i$ has the same upper bound as $W_i$'s. Repeating the precess finite times, we have
\[
|V_i(Y)|\le C (1+|Y|)^{4\sigma+2-n}.
\]
Therefore, we complete the proof.

\end{proof}

\begin{cor}\label{cor:sym} Assume as Lemma \ref{lem:s-1}. We have
\begin{align*}
&|\nabla^k(\Phi_i(y,0)-\theta_1 (y))|\le C (1+|Y|)^{-k}    \\&  \times \begin{cases} m_i^{-2}&\quad \mbox{if }4\sigma\le n<4\sigma+2 \\
\max\{\|\nabla^2 a_i\|_{L^\infty(B_1)}m_i^{-2+\frac{2\sigma}{n-2\sigma}}(1+|Y|)^{-\sigma}, m_i^{-2}\} &\quad \mbox{if }n=4\sigma+2,\\
\max\{ \|\nabla^2 a_i\|_{L^\infty(B_1)} m_i^{-2+\frac{2(n-(4\sigma+2))}{n-2\sigma}}(1+|Y|)^{4\sigma+2-n}, m_i^{-2}\}&\quad \mbox{if }n>4\sigma+2
\end{cases}
\end{align*}
for $k=0,1$, where $C>0$ depends only on $n,\sigma, A_0 $ and $A_1$.
\end{cor}

\begin{proof} Consider the equation of $\Phi_i -\Theta_i$, and the corollary follows from Lemma \ref{lem:s-1}, Lemma \ref{lem:s-1} and estimates of solutions to linear equations.

\end{proof}

\section{Estimates of Pohozaev integral for blow up solutions}

\label{sec:6}

\begin{prop}\label{prop:sign} Assume as Lemma \ref{lem:up-1}. Assume further that $\|a_i\|_{C^4(B_1)}\le A_0$.  Then for $0<r<\rho$ there holds
\begin{align*}
&m_i^{2}  P_\sigma(Y_i,r, U_i) =- m_i^{2} Q_\sigma(Y_i,r,U_i,p_i) \ge - C_0 r^{-n} m_i^{-\frac{4\sigma}{n-2\sigma}} -C_{0}\|a_{i}\|_{L^{\infty}(B_{1})}r^{4\sigma-n} \\ & \quad +
\frac{1}{C_1} \begin{cases}
\sigma a_i(y_i)\ln (rm^{\frac{2}{n-2\sigma}}_i) & \  n=4\sigma,\\
\sigma a_i(y_i)m^{\frac{2(n-4\sigma)}{n-2\sigma}}_{i} & \ 4\sigma< n<4\sigma+2,\\
\sigma a_i(y_i)m^{\frac{2(n-4\sigma)}{n-2\sigma}}_{i}+\frac{(\sigma+1)}{2n}\Delta a_{i}(y_{i})\ln (rm^{\frac{2}{n-2\sigma}}_i)  & \ n=4\sigma+2,\\
\sigma a_i(y_i)m^{\frac{2(n-4\sigma)}{n-2\sigma}}_{i}
+\frac{(\sigma+1)}{2n}\Delta a_{i}(y_{i})m^{\frac{2(n-4\sigma-2)}{n-2\sigma}}_{i}& 4\sigma+2<n<6\sigma+2,\\
\beta_{i}-C_{0}\|a_{i}\|_{B_{1}} \ln (rm^{\frac{2}{n-2\sigma}}_i)& \  n=6\sigma+2,\\
\beta_{i}-C_{0}\|a_{i}\|_{B_{1}} m_i^{\frac{2(n-6\sigma-2)}{n-2\sigma}}  & \  n>6\sigma+2,\\
\end{cases}
\end{align*}
where $\beta_i:=\sigma a_i(y_i) m_i^{\frac{2(n-4\sigma)}{n-2\sigma}} +\frac{(\sigma+1)}{2n}\Delta a_i(y_i) m_i^{\frac{2(n-4\sigma-2)}{n-2\sigma}}$, $\|a_i\|_{B_1}:=\|a_i\|_{L^\infty(B_1)}\|\nabla^2 a_i\|_{L^\infty(B_1)}+\|\nabla ^4 a_i\|_{L^\infty(B_1)}$, $C_0>0$ depends only on $n,\sigma, A_0,A_1$,$\rho$ and independent of $r$ if $i$ is sufficiently large, and $C_1>0$ depends only on $n$ and $\sigma$.
\end{prop}

\begin{proof}
By Proposition \ref{prop:pohozaev}, we have
\[
P_\sigma(Y_i,r, U_i)=-Q_\sigma(Y_i,r,U_i,p_i)=-\int_{B_r(y_i)}\big((y -y_i)_k\pa_k u_i+\frac{n-2\sigma}{2} u_i\big) a_i u_i\,\ud y +\mathcal{N}(r,u_i),
\]
where
\[
\mathcal{N}(r,u_i)= \frac{(n-2\sigma)\tau_i}{2 (p_i+1)}\int_{B_r(y_i)} u_i(y)^{p_i+1}\,\ud y -\frac{r}{p_i+1} \int_{\pa B_r(y_i)} u_i^{p_i+1}\,\ud S.
\]
By Proposition \ref{prop:up-1},
\be\label{eq:80}
m_i^2 \mathcal{N}(r,u_i) \ge - C r^{-n} m_i^{1-p_i}.
\ee

By change of variables $z=\ell_i (y-y_i)$ with $\ell_i=m_i^{\frac{p_i-1}{2\sigma}}$, we have
\begin{align*}
\mathcal{E}_i(r):&=-m_i^2\int_{B_r(y_i)}((y -y_i)_k\pa_k u_i+\frac{n-2\sigma}{2} u_i) a_i u_i\,\ud y\\&= -m_i^{4-\frac{n(p_i-1)}{2\sigma}}  \int_{B_{\ell_ir}}(z_k\pa_k \phi_i+\frac{n-2\sigma}{2} \phi_i) a_i(y_i+\ell_i^{-1}z) \phi_i\,\ud z,
\end{align*}
where $\phi_i(z)=m_i^{-1}u_{i}(\ell_{i}^{-1} z+y_i)$.
Let
\begin{align*}
\hat{\mathcal{E}}_i(r):= -m_i^{4-\frac{n(p_i-1)}{2\sigma}}  \int_{B_{\ell_ir}}(z_k\pa_k \theta_1+\frac{n-2\sigma}{2} \theta_1) a_i(y_i+\ell_i^{-1}z) \theta_1\,\ud z.
\end{align*}
Making use of Corollary \ref{cor:sym} and Lemma \ref{lem:tau-1}, we have
\begin{align*}
&|\mathcal{E}_i(r)-\hat{\mathcal{E}}_i(r)|\le C \|a_i\|_{L^\infty(B_1)}  m_i^{2-\frac{4\sigma}{n-2\sigma}} \int_{B_{\ell_ir}} \sum_{j=0}^1|\nabla ^{j} (\phi_i-\theta_1)|(z)(1+|z|)^{2\sigma-n+j}\,\ud z\\&
\le C \|a_i\|_{L^\infty(B_1)} \begin{cases}
r^{2\sigma}& \quad \mbox{if } n<4\sigma+2,\\
\max\{\|\nabla^2 a_i \|_{L^\infty(B_1)} r^{\sigma}, r^{2\sigma}\}  & \quad \mbox{if } n=4\sigma+2,\\
\max\{\|\nabla^2 a_i \|_{L^\infty(B_1)} r^{6\sigma+2-n}, r^{2\sigma}\}&  \quad \mbox{if } 4\sigma+2<n<6\sigma+2,\\
\max\{\|\nabla^2 a_i \|_{L^\infty(B_1)}\ln(rm_i^{\frac{2}{n-2\sigma}}), r^{2\sigma}\} & \quad \mbox{if } n=6\sigma+2,\\
\max\{\|\nabla^2 a_i \|_{L^\infty(B_1)} m_i^{\frac{2(n-6\sigma-2)}{n-2\sigma}}, r^{2\sigma}\} & \quad \mbox{if } n>6\sigma+2,\\
\end{cases}
\end{align*}
where $C>0$ depends only on $n,\sigma, A_0$ and $A_1$. Next, by direction computations we see that

\begin{align*}
\hat{\mathcal{E}}_i(r)&=-m_i^2\int_{B_r}(y_k\pa_k \theta_{\ell_i}+\frac{n-2\sigma}{2}  \theta_{\ell_i}) a_i(y_i+y)  \theta_{\ell_i} \,\ud y\\&
\ge m_i^2 \int_{B_r} (\frac12 y _k\pa_k a_i (y_i+y)+\sigma a_i(y_i+y)) \theta_{\ell_i}^2\,\ud y-C\|a_{i}\|_{L^{\infty}(B_{1})}r^{4\sigma-n} \\&
\ge   m_i^2 \int_{B_r}\Big (\sigma a_i(y_i) +(\sigma+\frac12) y_k\pa_k a_i(y_i)+ (\frac12+\frac{\sigma}{2} )\pa_{kl} a_i(y_i) y_ky_l\\&\qquad +(\frac14 +\frac{\sigma}{6})\pa_{jkl}a_i(y_i)y_j y_ky_l \Big)
\theta_{\ell_i}^2\,\ud y-C m_i^2\|\nabla^{4}a_{i}\|_{L^{\infty}(B_{1})} \int_{B_r} |y|^4\theta_{\ell_i}^2\,\ud y-C\|a_{i}\|_{L^{\infty}(B_{1})}r^{4\sigma-n}\\&
= m_i^2 \int_{B_r} (\sigma a_i(y_i)+\frac{(\sigma+1)}{2n} \Delta a_i(y_i) |y|^2)\theta_{\ell_i}^2\,\ud y\\&\qquad-C m_i^2\|\nabla^{4}a_{i}\|_{L^{\infty}(B_{1})} \int_{B_r} |y|^4\theta_{\ell_i}^2\,\ud y-C\|a_{i}\|_{L^{\infty}(B_{1})}r^{4\sigma-n},
\end{align*}
\begin{align*}
&m_i^2 \int_{B_r} (\sigma a_i(y_i)+\frac{(\sigma+1)}{2n} \Delta a_i(y_i) |y|^2)\theta_{\ell_i}^2\,\ud y \ge \\&
\frac{1}{C} \begin{cases}
\sigma a_i(y_i)\ln (rm^{\frac{2}{n-2\sigma}}_i) & \ n=4\sigma,\\
\sigma a_i(y_i)m^{\frac{2(n-4\sigma)}{n-2\sigma}}_{i} & 4\sigma<n<4\sigma+2,\\
\sigma a_i(y_i)m^{\frac{2(n-4\sigma)}{n-2\sigma}}_{i}+\frac{(\sigma+1)}{2n}\Delta a_{i}(y_{i})\ln (rm^{\frac{2}{n-2\sigma}}_i) & \  n=4\sigma+2,\\
\sigma a_i(y_i)m^{\frac{2(n-4\sigma)}{n-2\sigma}}_{i}
+\frac{(\sigma+1)}{2n}\Delta a_{i}(y_{i})m^{\frac{2(n-4\sigma-2)}{n-2\sigma}}_{i} &
n>4\sigma+2
\end{cases}
\end{align*}
and
\begin{align*}
& m_i^2 \int_{B_r} |y|^4\theta_{\ell_i}^2\,\ud y\le C  \begin{cases}
r^{4\sigma+4-n} & \ n<4\sigma+4\\
 \ln (rm^{\frac{2}{n-2\sigma}}_i)  & \  n=4\sigma+4,\\
m^{\frac{2(n-4\sigma-4)}{n-2\sigma}}_i  & \  n>4\sigma+4,
\end{cases}
\end{align*}
where $C>0$ depends only on $n,\sigma$ and $\sup_i\|\nabla ^4 a_i\|_{L^\infty(B_1)}$.
Since $4\sigma+4>6\sigma+2$ and
\[
-m_i^2Q_\sigma(Y_i,r,U_i,p_i)\ge \hat{\mathcal{E}}_i(r)-|\mathcal{E}_i(r)-\hat{\mathcal{E}}_i(r)|+ m_i^2 \mathcal{N}(r,u_i),
\]
the proposition follows immediately.

\end{proof}

By Proposition \ref{prop:up-1} and local estimates in \cite{JLX}, after passing to a subsequence we have
\[
|\pa_y^k(U_i(Y_i) U_i(Y)- U(Y))|+|s^{1-2\sigma}\pa_s(U_i(Y_i) U_i(Y)- U(Y))| \to 0 \quad \mbox{in } C^{\al}(\B_{1}^+\setminus \B_{\rho}^+)
\]
for $k=0,1,2$, some $\al\in (0,1)$ and all $\rho>0$, where $0\le U\in W^{1,2}(s^{1-2\sigma}, \B_1^+\setminus \B_{\rho}^+)$ for all $\rho>0$ satisfies
\begin{equation}\label{eq:toprove-1}
\begin{cases}
\mathrm{div}(t^{1-2\sigma}\nabla U)=0& \quad \mbox{in }\mathcal{B}_1^+,\\
\frac{\pa }{\pa \nu^\sigma} U=a U& \quad \mbox{on } \pa' \mathcal{B}^+_{1}\setminus\{0\}
\end{cases}
\end{equation}
in weak sense. We will still denote the subsequence as $U_i$. Notice that for every $0<r<1$
\be\label{eq:78}
m_i^2 P_\sigma(Y_i,r, U_i)\to P_\sigma(0,r, U) \quad \mbox{as }i\to \infty.
\ee

\begin{prop}\label{prop:is-iss} Assume as Lemma \ref{lem:s-harnack}. Suppose that for large $i$
\begin{itemize}
\item[(i)] $\beta_i\ge 0$ if $4\sigma+2\le n<6\sigma+2$;
\item[(ii)] $ \beta_i\ge (C_0+1) \| a_i\|_{B_1}\ln m_i$ if $n=6\sigma+2$;
\item[(iii)]  $ \beta_i \ge  (C_0+1) \| a_i\|_{B_1}m_i^{\frac{2(n-6\sigma-2)}{n-2\sigma}}  $  if $n>6\sigma+2$;
\end{itemize}
where $\beta_i:=\sigma a_i(y_i) m_i^{\frac{2(n-4\sigma)}{n-2\sigma}} +\frac{(\sigma+1)}{2n}\Delta a_i(y_i) m_i^{\frac{2(n-4\sigma-2)}{n-2\sigma}}$ if $n>4\sigma+2$ and $\beta_i:=\sigma a_i(y_i) m_i^{\frac{2(n-4\sigma)}{n-2\sigma}} +\frac{(\sigma+1)}{2n}\Delta a_i(y_i) \ln m_i$ if $n=4\sigma+2$, $m_i=u_i(y_i)$, and $C_0$ is the constant in Proposition \ref{prop:sign} with $\rho=1$,  then, after passing to a subsequence, $y_i\to 0$ is an isolated simple blow up point of $\{u_i\}$.
\end{prop}

\begin{proof} By Proposition \ref{prop:blowupbubble}, $r^{2\sigma/(p_i-1)}\overline u_i(r)$ has precisely
one critical point in the interval $0<r<r_i:=R_iu_i(y_i)^{-\frac{p_i-1}{2\sigma}}$.
If the proposition were wrong, let $ \mu_i\ge r_i$ be the second critical point of $r^{2\sigma/(p_i-1)}\overline u_i(r)$. Then there must hold
\be\label{eq:is-1}
\lim_{i\to \infty}\mu_i=0.
\ee

Without loss of generality, we assume that $y_i=0$. Set
\[
\phi_i(y):=\mu_i^{2\sigma/(p_i-1)}u_i(\mu_i y),\quad y\in \R^n.
\]
Clearly, $\phi_i$ satisfies
\[
\begin{split}
(-\Delta)^\sigma \phi_i&=\mu_i^{2\sigma}a_i(\mu_iy)\phi_i+\phi_i^{p_i},
\\
|y|^{2\sigma/(p_i-1)}\phi_i(y)&\leq \tilde{C},\quad |y|<1/\mu_i,
\\
\lim_{i\to \infty}\phi_i(0)&=\infty,
\end{split}
\]
\[
r^{2\sigma/(p_i-1)}\overline \phi_i(r)\mbox{ has precisely one critical point in } 0<r<1,
\]
and
\[
\frac{\mathrm{d}}{\mathrm{d}r}\left\{ r^{2\sigma/(p_i-1)}\overline \phi_i(r)\right\}\Big|_{r=1}=0,
\]
where $\overline \phi_i(r)=|\pa B_r|^{-1}\int_{\pa B_r}\phi_i$. Denote $\tilde a_i(y)$ by $\mu_i^{2\sigma}a_i(\mu_iy)$.

Therefore, $0$ is an isolated simple blow up point of $\phi_i$. Let $\Phi_i(Y)$ be the extension of $\phi_i(y)$ in the upper half space.
Then Lemma \ref{lem:s-harnack}, Proposition \ref{prop:up-1}, Proposition \ref{prop:bocher} and estimates for linear equations in \cite{JLX} imply that
\be \label{eq:60}
\Phi_i(0)\Phi_i(Y)\to G(Y)=A|Y|^{2\sigma-n}+H(Y)
\quad \mbox{in } C^{\al}_{loc}(\overline{\R^{n+1}_+}\setminus \{0\})\cap C^2_{loc}(\R^{n+1}_+),
\ee
and
\be\label{5.3new}
\phi_i(0)\phi_i(y)\to G(y,0)=A|y|^{2\sigma-n}+H(y,0)
\quad \mbox{in } C^2_{loc}(\R^{n}\backslash\{0\})
\ee
as $i\to \infty$, where $A>0$, $H(Y)$ satisfies
\[
\begin{cases}
\mathrm{div}(t^{1-2\sigma}\nabla H)=0\quad &\mbox{in }\R^{n+1}_+,\\
\frac{\pa }{\pa \nu^\sigma} H(y,0)=0\quad &\mbox{for } y\in \R^{n},
\end{cases}
\]
in weak sense.

Note that $G(Y)$ is nonnegative, we have $\liminf_{|Y|\to \infty}H(Y)\geq 0$. It follows from the weak maximum principle and the
Harnack inequality that $H(y)\equiv H\geq 0$ is a constant. Since
\[
\frac{\mathrm{d}}{\mathrm{d}r}\left\{ r^{2\sigma/(p_i-1)}\phi_i(0)\overline \phi_i(r)\right\}\Big|_{r=1}=
\phi_i(0)\frac{\mathrm{d}}{\mathrm{d}r}\left\{ r^{2\sigma/(p_i-1)}\overline \phi_i(r)\right\}\Big|_{r=1}=0,
\]
we have, by sending $i$ to $\infty$ and making use of \eqref{5.3new}, that
\[A=H>0.\]

By \eqref{eq:60} and the interior estimates for linear equation in \cite{JLX}, we have
\be \label{eq:sign-1}
\liminf_{i\to \infty}\Phi_i(0)^2 P_\sigma(0,\delta, \Phi_i)= P_\sigma(0,\delta, G)=-\frac{(n-2\sigma)^2}{2}A^2\int_{\pa''\B_1^+}t^{1-2\sigma}< 0.
\ee

If $n<6\sigma+2$, by Proposition \ref{prop:sign} and item (i) in the assumptions we have
\[
\liminf_{\delta\to 0}\liminf_{i\to \infty}  \Phi_i(0)^2 P_\sigma(0,\delta, \Phi_i) \ge 0.
\]
This contradicts to \eqref{eq:sign-1}. Hence $y_i\to 0$ has to be an isolated simple blow up point of $\{u_i\}$ upon passing to a subsequence.

If $n\ge 6\sigma+2$, let
\begin{align*}
\tilde \beta_i:&=\sigma \tilde a_i(y_i) \Phi_i(0)^{\frac{2(n-4\sigma)}{n-2\sigma}} +\frac{(\sigma+1)}{2n}\Delta \tilde a_i(y_i) \Phi_i(0)^{\frac{2(n-4\sigma-2)}{n-2\sigma}}\\&
=(1+o(1)) \mu_i^{n-2\sigma} \beta_i.
\end{align*}
Since $\|\tilde a_i\|_{C^4(B_1)}\le \mu_{i}^{2\sigma} A_0$ and $\|\tilde a_i\|_{B_1}\le \mu_i^{4\sigma+2} \|a_i\|_{B_1} $, we have
\begin{align}
&\tilde \beta_i -C_0\|\tilde a_i\|_{B_1} \ln \Phi_i(0)-C_{0}\|\tilde{a}_{i}\|_{L^{\infty}(B_{1})}\delta^{4\sigma-n} \nonumber \\&
\ge  (1+o(1)) \mu_i^{4\sigma+2} \beta_i -C_0 \mu_i^{4\sigma+2} \|a_i\|_{B_1}  \ln m_i-C_0\|\tilde a_i\|_{B_1}\ln \mu_i^{\frac{2\sigma}{p_i-1}} -C_{0}\mu^{2\sigma}_{i}\|a_{i}\|_{L^{\infty}(B_{1})}\delta^{4\sigma-n}\nonumber \\&
\ge  (1+o(1)) \mu_i^{4\sigma+2} (C_0+1) \|a_i\|_{B_1}\ln m_i -C_0 \mu_i^{4\sigma+2} \|a_i\|_{B_1}  \ln m_i-C_0\|\tilde a_i\|_{B_1}\ln \mu_i^{\frac{2\sigma}{p_i-1}} \nonumber  \\&
~~-C_{0}\mu^{2\sigma}_{i}\|a_{i}\|_{L^{\infty}(B_{1})}\delta^{4\sigma-n}
\ge 0 \label{eq:sl-1}
\end{align}
for $n=6\sigma+2$ and
\begin{align}
&\tilde \beta_i -C_0\|\tilde a_i\|_{B_1} \Phi_i(0)^{\frac{2(n-6\sigma-2)}{n-2\sigma}}-C_{0}\|\tilde{a}_{i}\|_{L^{\infty}(B_{1})}\delta^{4\sigma-n} \nonumber
\\& \ge  (1+o(1)) \mu_i^{n-2\sigma} \beta_i -C_0(1+o(1)) \mu_i^{n-2\sigma} \|a_i\|_{B_1}   m_i^{\frac{2(n-6\sigma-2)}{n-2\sigma}}-C_{0}\mu^{2\sigma}_{i}\|a_{i}\|_{L^{\infty}(B_{1})}\delta^{4\sigma-n}\nonumber
 \nonumber\\&
\ge  (1+o(1)) \mu_i^{n-2\sigma} (C_0+1)  \|a_i\|_{B_1} m_i^{\frac{2(n-6\sigma-2)}{n-2\sigma}} -C_0(1+o(1)) \mu_i^{n-2\sigma} \|a_i\|_{B_1}   m_i^{\frac{2(n-6\sigma-2)}{n-2\sigma}}\nonumber\\&
~~-C_{0}\mu^{2\sigma}_{i}\|a_{i}\|_{L^{\infty}(B_{1})}\delta^{4\sigma-n}
\ge 0
\label{eq:sl-2}
\end{align}
for $n> 6\sigma+2$. By Proposition \ref{prop:sign}
\[
\liminf_{\delta\to 0}\liminf_{i\to \infty}  \Phi_i(0)^2 P_\sigma(0,\delta, \Phi_i) \ge 0.
\]
This contradicts to \eqref{eq:sign-1}. Hence $y_i\to 0$ has to be an isolated simple blow up point of $\{u_i\}$ upon passing to a subsequence.

Therefore, we complete the proof of Proposition \ref{prop:is-iss}.

\end{proof}

\begin{rem}\label{rem:confirm}Note that

\begin{enumerate}
\item From \eqref{eq:sl-1} and \eqref{eq:sl-2}, we say assumptions (ii) and (iii) in Proposition \ref{prop:is-iss} are scaling invariant.
\item Either $a_i>\frac{1}{C}$ in $B_1$ for some $C>0$ or $a_i\ge 0$ and $\Delta a_i\ge 1/C $  on $\{x:a_i(x)<d\}\cap B_2$ for some constant $d>0$ when  $n\ge 4\sigma+2$,  then the assumptions (ii) and (iii) in Proposition \ref{prop:is-iss} hold automatically.
\end{enumerate}
\end{rem}

\section{Proof of the main theorems}

\label{sec:7}

Let $u\in C^2(B_3)\cap \mathcal{L}_\sigma(\R^n)$ be a solution of
\be \label{eq:mode}
(-\Delta )^\sigma u-a(x) u=u^{p} \quad \mbox{in }B_3, \quad u>0 \quad \mbox{in }\R^n,
\ee
where $a(x)\in C^2(B_3)$ and $1<p\le \frac{n+2\sigma}{n-2\sigma}$.

\begin{prop}\label{prop:bubledec} Assume as above. Then for any $0<\va<1$ and $R>1$, there exists large positive constants $C_1 $ and $ C_2$ depending only on $n,\sigma, \|a\|_{C^2(B_2)}$, $\va$ and $R$ such that the following statement holds. If
\[
\max_{\bar B_2} dist(x,\pa B_2)^{\frac{n-2\sigma}{2}} u(x)\ge C_1,
\]
then $p\ge \frac{n+2\sigma}{n-2\sigma}-\va$ and a finite set $S$ of local maximum points of $u$ in $B_2$ such that:
\begin{itemize}
\item[(i).] For any $y\in S$, it holds
\[
\|u(y)^{-1}u(u(y)^{\frac{p-1}{2\sigma}} x+y)-(1+|x|^2)^{\frac{2\sigma-n}{2}}\|_{C^2(B_{2R})} < \va,
\]
where $\bar c>0$ depends only on $n,\sigma$.
\item[(ii).] If $y_1,y_2\in S$ and $y_1\neq y_2$, then
\[
B_{R u(y_1)^{(1-p)/2\sigma}}(y_1) \cap B_{R u(y_2)^{(1-p)/2\sigma}}(y_2) =\emptyset.
\]
\item[(iii).] $u(x)\le C_2 dist(x,S)^{-2\sigma/(p-1)}$ for all $x\in B_2$.
\end{itemize}

\end{prop}

The proof is standard by now, which follows from the blow-up argument as the proof of Proposition \ref{prop:blowupbubble} and Liouville theorem in Jin-Li-Xiong \cite{JLX}. We omit it here.

\begin{prop}\label{prop:energy}  Let  $u\in C^2(B_3)\cap \mathcal{L}_\sigma(\R^n)$ be a solution of \eqref{eq:mode} with $0\le a\in C^4(B_3)$. Suppose that  $\Delta a\ge 0 $  on $\{x:a(x)<d\}\cap B_2$ for some constant $d>0$, and further that  $\Delta a>\gamma>0 $  on $\{x:a(x)<d\}\cap B_2$ for some constant $\gamma$ if $ n\ge 6\sigma+2$.
Then for any $\va>0$ and $R>1$, once $\max_{\bar B_2}dist(x,\pa B_2)^{\frac{n-2\sigma}{2}}u(x)\ge C_1$ with the constant $C_1$ given by Proposition \ref{prop:bubledec} there must be true
\[
|y_1-y_2|\ge \delta^*>0 \quad \mbox{for every }y_1,y_2 \in S \cap B_{3/2},
\]
where $S$ associated to $u$ is also given by  Proposition \ref{prop:bubledec}, and the constant $\delta^*$ depends only on $n,\sigma,d,\gamma, \va, R$ and $\|a\|_{C^4(B_3)}$.
\end{prop}
\begin{proof} The idea is  similar to that of Proposition 5.2 of \cite{JLX} on the unit sphere, but Lemma \ref{lem:select} have to be used since our equation is defined in a bounded domain with boundary. Suppose the contrary, for some $\va, R$ and $d>0$, there exist sequence $\{p_i\}$ and nonnegative potentials  $ a_i\to {\color{red}a}$ in $ C^4(B_3)$ with $\|a_i\|_{C^4(B_3)}\le A_0 $, satisfying the assumptions for $a$, and   a sequence of corresponding solutions $\{u_i\}_{i=1}^\infty$ such that
\[
\lim_{i\to \infty} \min_{j\neq l} |z_{i,j}-z_{i,l}|=0,
\]
where $z_{i,j},z_{i,l}\in S_i\cap B_{3/2}$ associated to $u_i$ defined in Proposition \ref{prop:bubledec}.

 Upon passing to a subsequence, we assume $
z_{i,j},z_{i,l}\to \bar z\in \bar B_{3/2}.$ Define $f_i(z):S_i\to (0,\infty)$ by $f_i(z)=\min_{y\in S_i\setminus \{ z\}}|z-y|$. Let $R_i\to \infty$ with $R_if_i(z_{i,j})\to 0$ as $i\to \infty$.
By Lemma \ref{lem:select}, one can find, say,  $z_{i,1}\in S_i\cap B_{2R_i f_{i}(z_{i,j})}(z_{i,j})$ satisfying
\[
f_i(z_{i,1})\le (2R_i+1) f_{i}(z_{i,j}) \quad \mbox{and} \quad
\min_{z\in  S_i\cap B_{R_i f_i(z_{i,1})} (z_{i,1})} f_i(z) \ge \frac12 f_i(z_{i,1}).
\]
Let $|z_{i,2}-z_{i,1}|=f(z_{i,1})$. Let $U_i$ be the extension of $u_i$ and
\[
\Phi_i(X)=f_i(z_{i,1})^{2\sigma/(p_i-1)}U_i(f_i(z_{i,1}) X+Z_{i,1}) \quad \mbox{with }Z_{i,1}=(z_{i,1},0).
\]
The rest of the proof is divided into three steps:
\begin{enumerate}
\item Prove that $0$ and $x_i:=f_i(z_{i,1})^{-1}(z_{i,2}-z_{i,1})\to \bar x$ with $|\bar x|=1$ are two isolated blow up points of $\{\Phi_i(x,0)\}$.

\item By Proposition \ref{prop:is-iss}, after passing to a subsequence $0$ and $x_i\to \bar x$ have to be isolated simple blow up points of $\{\Phi_i(x,0)\}$.

\item Since $\Phi_i(0)\Phi_i(X)$ tends to a Green function with at least two poles,
we can drive a contradiction by Pohozaev identity.
\end{enumerate}

For step 1 and 3, see the proof of Proposition 5.2 of \cite{JLX}. For step 2, we let
\[
\tilde a_i(x):= f_i(z_{i,1})^{2\sigma} a_i(f_i(z_{i,1}) x+z_{i,1})
\] and verify assumptions in Proposition \ref{prop:is-iss}. We only show it if $n\ge 6\sigma+2$.  By the assumption of $a_i$,  we have
\begin{align*}
&\frac{\sigma \tilde  a_i(0)\Phi_i(0)^{\frac{2(n-4\sigma)}{n-2\sigma}} +\frac{\sigma+1}{2n}\Delta \tilde a_i(0)\Phi_i(0)^{\frac{2(n-4\sigma-2)}{n-2\sigma}}}{\|\tilde a_i\|_{B_{1/2}}}\\&
\ge \Phi_i(0)^{\frac{2(n-4\sigma-2)}{n-2\sigma}}\|a_i\|_{B_2}^{-1}\big\{\sigma a_i(z_{i,1})\Phi_i(0)^{\frac{4}{n-2\sigma}}f_i(z_{i,1})^{-2} +\frac{\sigma+1}{2n} \Delta a_i(z_{i,1})\Big\}\\&
\ge \Phi_i(0)^{\frac{2(n-4\sigma-2)}{n-2\sigma}}\|a_i\|_{B_2}^{-1} \frac{\sigma+1}{2n} \cdot \gamma \quad \mbox{for large }i.
\end{align*} Since
\[
\frac{ \Phi_i(0)^{\frac{2(n-4\sigma-2)}{n-2\sigma}}\|a_i\|_{B_2}^{-1} \frac{\sigma+1}{2n} \cdot \gamma}{\ln \Phi_i(0)}\to \infty \quad \mbox{if }n=6\sigma+2
\]
and
\[
\frac{ \Phi_i(0)^{\frac{2(n-4\sigma-2)}{n-2\sigma}}\|a_i\|_{B_2}^{-1} \frac{\sigma+1}{2n} \cdot \gamma}{\Phi_i(0)^{\frac{2(n-6\sigma-2)}{n-2\sigma}}}\to \infty \quad \mbox{if }n>6\sigma+2,
\]
by Proposition \ref{prop:is-iss} $0$ is an isolated simple blow up point of $\{\Phi_i(\cdot,0)\}$. Similarly, one can show $x_i\to \bar x$ is an isolated simple blow up point of $\{\Phi_i(\cdot,0)\}$.

 Therefore, we complete the proof of  Proposition \ref{prop:energy}.

\end{proof}

\begin{proof}[Proof of Theorem \ref{thm:2}] We first prove that $\|u\|_{L^\infty(B_{5/4})}\le C$. Suppose the contrary that there exists a sequence of solutions $u_i$ of \eqref{eq:main1} satisfying $\|u_i\|_{L^\infty(B_{5/4})}\to \infty$ as $i\to \infty$. For any fixed $\va>0$ sufficiently small and $R>>1$, by Proposition \ref{prop:energy} the set  $S_i$ associated to $u_i$ defined by Proposition \ref{prop:bubledec} only consists of finite many points in $B_{3/2}$ with a uniform positive lower bound of distances between each two points, if $S_i\cap B_{3/2}$ has points more than $1$. By the contradiction assumption $\|u_i\|_{L^\infty(B_{5/4})}\to \infty$ and Proposition \ref{prop:bubledec}, $S_i\cap B_{11/8}$ is not empty and  has only isolated blow up points of $\{u_i\}$ after passing to a subsequence. By Proposition \ref{prop:is-iss}, these isolated blow up points have to be isolated simple blow up points. Suppose that $y_i\to \bar y\in \bar B_{11/8}$ is an isolated simple blow up point of $\{u_i\}$. Let $U_i$ be the extensions of $u_i$ and $Y_i=(y_i,0)$.  By Proposition \ref{prop:up-1}, we have
\[
|U_i(Y_i)^2 P_\sigma(Y_i, r, U_i)| \le C(r).
\]
On the other hand, by the assumption of $a$ and Proposition \ref{prop:sign} we have \[
\liminf_{i\to \infty} U_i(Y_i)^2 P_\sigma(Y_i, r, U_i)=\infty \quad \mbox{for some small }r>0
 \]if $n\ge 4\sigma$. Hence, we obtain a contraction and thus $\|u\|_{L^\infty(B_{5/4})}\le C$. The theorem then follows from interior estimates of solutions of linear equations in \cite{JLX}.
\end{proof}

\begin{proof}[Proof of Theorem \ref{thm:3}]  For any fixed $\va>0$ sufficiently small and $R>>1$ let $S_i$ be the set associated to $u_i$ defined by Proposition \ref{prop:bubledec}.

If $4\sigma+2\le n<6\sigma+2$, by Proposition \ref{prop:energy} the set  $S_i$ associated to $u_i$ defined by Proposition \ref{prop:bubledec} only consists of finite many points in $B_{3/2}$. Since $u_i(x_i)\to \infty$ and $x_i\to \bar x $, by item (iii) of  Proposition \ref{prop:bubledec}, after passing to subsequence, there exists $S_i\ni x_i'\to \bar x$ is an isolated blow up point of $\{u_i\}$. By Proposition \ref{prop:is-iss}, it has to be an isolated simple blow up point.  Let $U_i$ be the extensions of $u_i$ and $X_i'=(x_i',0)$.  By Proposition \ref{prop:up-1}, we have
\[
|U_i(X_i')^2 P_\sigma(X_i', r, U_i)| \le C(r).
\]
By Proposition \ref{prop:sign}, we establish the theorem for $4\sigma+2\le n<6\sigma+2$.

If $n\ge 6\sigma+2$, suppose the contrary that, for some subsequence which we still denote as $i$,
\be \label{eq:contr-3}
\begin{split}
\sigma a_i(x_i) u_{i}(x_i)^{\frac{4}{n-2\sigma}}&+\frac{\sigma+1}{2n} \Delta a_i(x_i)\\ & \ge \frac{1}{|o(1)|} \begin{cases} u_i(x_i)^{\frac{4\sigma}{n-2\sigma}} \ln u_i(x_i)^{-1}& \quad \mbox{for }n=6\sigma+2,\\
u_i(x_i)^{\frac{4\sigma}{n-2\sigma}} & \quad \mbox{for }n>6\sigma+2.
\end{cases}
\end{split}
\ee
Let $\mu_i=dist\{x_i,S_i\setminus \{x_i\}\}$ and
\[
\Phi_i(X)=\mu_i^{\frac{n-2\sigma}{2}} U_i(\mu_i X+X_i),
\]
where $U_i$ is the extension of $u_i$ and $X_i=(x_i,0)$.  If $x_i\notin S_i$, we have $u_i(x_i)\le C \mu_i^{-\frac{n-2\sigma}{2}}$. Hence, $\Phi_i(0)\le C<\infty$ and $\mu_i\to 0$. Since  $\max\limits_{B_{\bar d}(x_i)}u_i(x)\le \bar b u_i(x_i)$, $\Phi_i(x,0)\le C\bar b$ for all $|x|\le \bar d/\mu_i$. By  the argument of proof of Proposition \ref{prop:blowupbubble}, for some $x_0\in \R^n$ and $\lda>0$,
 \[
\Phi_i(x,0)\to (\frac{\lda}{1+\bar c \lda^2 |x-x_0|^2})^{\frac{n-2\sigma}{2}} \quad \mbox{in }C^2_{loc}(\R^n).
\] Note that the limiting function has only one critical point. Suppose $z_i\in S_i$ satisfying $|z_{i}-x_i|=\mu_i$.  Since $x_i$ and $z_i$ both are local maximum points of $\{u_i\}$, $\nabla \Phi_i(0)=0$ and, after passing to subsequence,
\[
\frac{z_i-x_i}{\mu_i}\to \bar x \mbox{ with }|\bar x|=1, \quad 0= \nabla_x \Phi_i(\frac{z_i-x_i}{\mu_i},0).
\]
We obtain a contradiction. Hence, $x_i\in S_i$. It follows that $0$ is an isolated blow up point of $\{\Phi_i(x,0)\}$. By Remark \ref{rem:confirm} and contradiction assumption \eqref{eq:contr-3}, $0$ is an isolated simple blow up point. Making use of Proposition \ref{prop:up-1} and Proposition \ref{prop:sign} we obtain contradiction again.

Therefore, we complete the proof.

\end{proof}

\bigskip

\noindent  School of Mathematical Sciences, Beijing Normal University\\
Beijing 100875, China\\[1mm]
 Email: \textsf{miaomiaoniu@mail.bnu.edu.cn} (M.N);

 ~~~~~\textsf{201321130137@mail.bnu.edu.cn} (Z.P);

~~~~~\textsf{jx@bnu.edu.cn} (J.X)

\end{document}